\documentclass[11pt,a4paper]{amsart}
\usepackage{amsfonts}

\usepackage{}
\usepackage[mathscr]{eucal}
\usepackage{amssymb}
\usepackage{amsbsy}

\usepackage{CJK,CJKnumb}
\usepackage[CJKbookmarks,colorlinks,
            linkcolor=black,
            anchorcolor=black,
            citecolor=blue]{hyperref}
\usepackage{color,soul}              
\usepackage{indentfirst}        
\usepackage{latexsym,bm}        
\usepackage{graphics,graphicx}
\usepackage{cases}
\usepackage{pifont}
\usepackage{txfonts}
\usepackage{xcolor}
\usepackage[all,knot,poly,cmtip]{xy}
\usepackage[usenames,dvipsnames]{pstricks}
\usepackage{epsfig}
\usepackage{pst-grad} 
\usepackage{pst-plot} 
\usepackage{shuffle}

\allowdisplaybreaks[4]  

\numberwithin{equation}{section}

\newcommand{\al}{\alpha}
\newcommand{\be}{\beta}
\newcommand{\de}{\delta}
\newcommand{\De}{\Delta}
\newcommand{\ep}{\epsilon}

\newcommand{\ga}{\gamma}
\newcommand{\Ga}{\Gamma}
\newcommand{\la}{\lambda}

\newcommand{\La}{\Lambda}
\newcommand{\ot}{\otimes}

\newcommand{\Om}{\Omega}

\newcommand{\te}{\theta}
\newcommand{\Te}{\Theta}
\newcommand{\ve}{\varepsilon}

\newcommand{\mA}{\mathcal{A}}

\newcommand{\mC}{\mathcal{C}}
\newcommand{\mE}{\mathcal{E}}
\newcommand{\mK}{\mathcal{K}}
\newcommand{\mL}{\mathcal{L}}
\newcommand{\mP}{\mathcal{P}}

\newcommand{\mW}{\mathcal{W}}

\newcommand{\bN}{\mathbb{N}}
\newcommand{\bQ}{\mathbb{Q}}
\newcommand{\bP}{\mathbb{P}}
\newcommand{\bZ}{\mathbb{Z}}

\newcommand{\mb}[1]{\mbox{#1}}

\newcommand{\bs}[1]{{\scriptsize\mbox{#1}}}

\newcommand{\stt}[1]{{\scriptstyle #1}}

\newcommand{\ol}[1]{\overline{#1}}

\newcommand{\lan}{\langle}
\newcommand{\ran}{\rangle}
\newcommand{\lb}{\left(}
\newcommand{\rb}{\right)}
\newcommand{\rw}{\rightarrow}
\newcommand{\beq}{\begin{equation}}
\newcommand{\eeq}{\end{equation}}

\begin{document}

\newtheorem{theorem}{Theorem}[section]

\newtheorem{lem}[theorem]{Lemma}

\newtheorem{cor}[theorem]{Corollary}
\newtheorem{prop}[theorem]{Proposition}

\theoremstyle{remark}
\newtheorem{rem}[theorem]{Remark}

\newtheorem{defn}[theorem]{Definition}

\newtheorem{exam}[theorem]{Example}

\theoremstyle{conjecture}
\newtheorem{con}[theorem]{Conjecture}

\setcounter{section}{-1}

\renewcommand\arraystretch{1.2}

\title[On weak peak quasisymmetric functions]{On weak peak quasisymmetric functions}

\author[Li]{Yunnan Li}
\address{School of Mathematics and Information Science, Guangzhou University, Waihuan Road West 230, Guangzhou 510006, China}
\email{ynli@gzhu.edu.cn}


\begin{abstract}
In this paper, we construct the weak version of peak quasisymmetric functions inside the Hopf algebra of weak composition quasisymmetric functions (WCQSym) defined by Guo, Thibon and Yu. Weak peak quasisymmetric functions (WPQSym) are studied in several aspects. First we find a natural basis of WPQSym lifting Stembridge's peak functions. Then we confirm that WPQSym is a Hopf subalgebra of WCQSym by giving explicit multiplication, comultiplication and antipode formulas. By extending Stembridge's descent-to-peak maps, we also show that WPQSym is a Hopf quotient algebra of WCQSym. On the other hand, we prove that WPQSym embeds as a Rota-Baxter subalgebra of WCQSym, thus of the free commutative Rota-Baxter algebra of weight 1 on one generator. Moreover,
WPQSym can also be a Rota-Baxter quotient algebra of WCQSym.
\end{abstract}


\maketitle


\section{Introduction}

Weak compositions involving zeros have many applications in combinatorics. However, to our knowledge the study on properties of weak compositions can rarely be found in literature. Recently, the term of ``weak composition quasisymmetric functions'' is proposed, and the algebra of weak composition quasisymmetric functions (WCQSym) is thoroughly studied in \cite{GTY}. It generalizes the ordinary quasisymmetric functions (QSym) defined by Gessel in \cite{Ge} with many extraordinary properties, especially Hopf algebra structures; see \cite{GKL,LMW,MR}. By comparison with QSym, WCQSym is also a connected, graded Hopf algebra, but beyond the category of combinatorial Hopf algebras considered in \cite{ABS,BLL,BL}, as its graded components are no longer finite dimensional.

One of the interesting applications to consider weak composition quasisymmetric functions is due to the study about Rota-Baxter algebras; see also \cite{GTY,YGZ}.
The discovery of Rota-Baxter algebras originates from Baxter's probability research to understand Spitzer's identity in fluctuation theory. So far the theory of Rota-Baxter algebras have been developed in many areas involving mathematics and theoretical physics; see \cite{Guo,KRY} for a detailed introduction to this subject.

In \cite{GK} the authors construct free Rota-Baxter algebras by using the mixable shuffle algebra.
In \cite{EG} the authors discuss Hopf algebra structures on free Rota-Baxter algebras, by relatively considering the mixable shuffle Hopf algebras, thus generalize a previous work \cite{AGKO} on this topic.
In \cite{YGZ} the authors have realized the free commutative \textit{nonunitary} Rota-Baxter algebras on one generator in terms of left weak composition quasisymmetric functions (LWCQSym). In order to further investigate the relationship between Rota-Baxter algebras and generalized symmetric functions envisioned by Rota after his work \cite{Ro}, Guo, Thibon and Yu in \cite{GTY} introduces weak composition quasisymmetric functions as formal power series with suitable semigroup exponents. In this paper, they focus on the free commutative unitary Rota-Baxter algebra $\shuffle(x)$ of weight 1 generated by one element $x$. First the Hopf algebra structure on WCQSym is considered, then the mixable shuffle algebra $\shuffle^+(x)$ is realized as the Hopf algebra WCQSym. Meanwhile, $\shuffle^+(x)$ can be straightly extended to define $\shuffle(x)$ as tensor product $\bZ[x]\ot\shuffle^+(x)$ of Hopf algebras.

On the other hand, there is an interesting class of quasisymmetric functions, peak quasisymmetric functions (PQSym). It is introduced independently by Billey and Haiman in \cite{BH} via Schubert calculus, and also by Stembridge in \cite{Ste} via the theory of enriched $P$-partitions. The study of these functions has close relation with many other topics in mathematics, such as representation theory \cite{BHT}, Hopf algebras \cite{BMSW,Sch}, Coxeter groups \cite{BB,BC}, discrete geometry \cite{BHW} and combinatorial tableaux \cite{JL,Og}, etc.

Mainly inspired by the work in \cite{GTY,Ste}, we apply the machinery of enriched $P$-partitions to the case of weak compositions, and define our weak version of peak quasisymmetric functions. By contrast with the classical peak functions, the weak peak quasisymmetric functions (WPQSym) have not only many parallel properties but also considerable new features, especially the connection with Rota-Baxter algebras. We lift Stembridge's peak functions onto WCQSym to define the weak peak fundamental quasisymmetric functions (WPFQF). The expansion formulas of WPFQFs in terms of two usual bases of WCQSym are given. By elaborate analysis, we extract a natural basis of WPQSym from WPFQFs, whose linear independence is tough to confirm.  Also we derive the multiplication rule of WPFQFs. Furthermore, we clarify the Hopf subalgebra structure of WPQSym in WCQSym. In particular, the cancellation-free antipode formula is proved by the method of sign-reversing involutions proposed in \cite{BS}. Besides, WPQSym has a prominent property that it is a Rota-Baxter subalgebra of WCQSym.

Note that zeros in weak compositions strongly effect WCQSym, and the proper definition of peak sets of weak compositions play a crucial role for our construction of WPQSym. Consequently, we fully study the Hopf subalgebra WPQSym$^0$ of WPQSym, with base functions indexed by weak compositions only having zeros. During this study, we find some unusual combinatorial identities as byproducts. Furthermore, we define a Hopf surjection from WCQSym to WPQSym, as an extension of Stembridge's descent-to-peak map, and thus WPQSym also becomes a Rota-Baxter quotient algebra of WCQSym.

The organization of the paper is as follows. In $\S 1$ we introduce some combinatorial notations, definitions, then recall weak composition quasisymmetric functions and peak quasisymmetric functions. In $\S 2$ we construct weak peak quasisymmetric functions, and give two expansion formulas for the WPFQFs. In $\S 3$ we mainly show that WPQSym is a graded subring of WCQSym, with the explicit multiplication rule of WPFQFs. A natural basis of WPQSym is obtained from WPFQFs. In $\S 4$ we discuss the Hopf algebra structure on WPQSym. The formulas of comultiplication and antipode for WPFQFs are given. Besides, we provide a Hopf algebra projection from WPQSym to PQSym, and extend Stembridge's descent-to-peak map to show that WPQSym is also a Hopf quotient algebra of WCQSym. In $\S 5$ we further show that WPQSym is not only a Rota-Baxter subalgebra of WCQSym but also a quotient algebra of WCQSym compatible with another Rota-Baxter algebra structure. In the end, we recap all the structures in our framework of weak peak quasisymmetric functions with a commutative diagram.

\section{Background}
\subsection{Notations and definitions}
Throughout this paper, we mainly consider two base rings $\bZ$ and $\bQ$.
For two rings $A$ and $k$,  we always denote $A_k:=A\ot_\bZ k$, when the base ring changes from $\bZ$ to $k$.

Denote by $\bN$ (resp. $\bZ'$) the set of nonnegative (resp. nonzero) integers, $\bP$ the set of positive integers. Given any $m,n\in\bP,\,m\leq n$, let $[m,n]:=\{m,m+1,\dots,n\}$ and $[n]:=[1,n]$ for short. Given $n\in\bN$, denote $\mW\mC(n)$ the set of \textit{weak compositions} of $n$, consisting of ordered tuples of nonnegative integers summed up to $n$. Let $\mW\mC:=\bigcup\limits_{n\geq0}^.\mW\mC(n)$. For convenience we assume $\emptyset\in\mW\mC$, called the \textit{empty} weak composition. In particular, we call $\al$ a \textit{composition} if all its entries are positive. Given $n\in\bP$, let $\mC(n)$ be the set of compositions of $n$ and $\mC:=\bigcup\limits_{n\geq1}^.\mC(n)$ with $\emptyset\in\mC$. Write $\al\vDash n$ when $\al\in \mC(n)$.
Assume that $\al=(0^{i_1},s_1,\dots,0^{i_k},s_k,0^{i_{k+1}})$ with all $i_p\in\bN$ and $s_q\in\bP$,
let $\ell(\al):=k+i_1+\cdots+i_{k+1}$ as the \textit{length} of $\al$, $\ell_0(\al):=i_1+\cdots+i_{k+1}$, called the \textit{0-length} of $\al$, $|\al|:=n$ as the \textit{weight} of $\al$, also $\|\al\|:=|\al|+\ell_0(\al)$, called the \textit{total weight} of $\al$, and define the \textit{descent set} of $\al$ as
\[D(\al):=\{a_1,\dots,a_k\,|\,a_j=i_1+s_1+\cdots+i_j+s_j,\,j=1,\dots,k\}.\]
Also, the \textit{refining order} $\leq$ on $\mW\mC(n)$ is defined as in \cite{GTY}.
If $\al=(0^{i_1},\al_1,\dots,0^{i_k},\al_k,0^{i_{k+1}})$, $\be=(0^{j_1},\be_1,\dots,0^{j_k},\be_k,0^{j_{k+1}})\in\mW\mC(n)$, where $i_1,\dots,i_{k+1},j_1,\dots,j_{k+1}\in\bN$ such that either $i_{k+1}=j_{k+1}=0$ or $i_{k+1},j_{k+1}\geq1$, and for any $q\in[k]$,
$\al_q,\be_q$ are compositions of the same weight, set
\[\al\leq\be\mb{ if }i_p\leq j_p\mb{ and }D(\be_q)\subseteq D(\al_q)\mb{ for }p=1,\dots,k+1\mb{ and }q=1,\dots,k.\]
For any weak composition $\al$, let $\ol{\al}$ be the composition obtaining from $\al$ by erasing all the zeros. For any $\al=(a_1,\dots,a_r),\,\be=(b_1,\dots,b_s)$ with all $a_i,b_j\in\bN$, define the \textit{concatenation} $\al\be:=(a_1,\dots,a_r,b_1,\dots,b_s)$, and the \textit{near concatenation} $\al\vee\be:=(a_1,\dots,a_r+b_1,\dots,b_s)$. Define the \textit{reverse} of $\al$ as $\al^r:=(a_r,\dots,a_1)$, the \textit{complement} $\al^c$ of $\al$ as
\[\al^c:=(1^{a_1})\bullet(1^{a_2})\bullet\cdots\bullet(1^{a_r}),\]
where we set $(1^{a_i})=(0)$ if $a_i=0$ and $(1^{a_i})\bullet(1^{a_{i+1}})$ means $(1^{a_i})\vee(1^{a_{i+1}})$ if $a_i,a_{i+1}\in\bP$, otherwise we only concatenate them.
Furthermore, the \textit{conjugate} or \textit{transpose} of $\al$ is $\al^t:=(\al^r)^c=(\al^c)^r$. For example, if $\al=(2,0^2,1^2,2)$, then $\al^c=(1^2,0^2,3,1)$ and $\al^t=(1,3,0^2,1^2)$. Sometimes for simplicity, we will express weak compositions as one-line words on $\bN$ without confusion.


\subsection{Weak composition quasisymmetric functions}
For $\al=(n_1,\dots,n_\ell)\in\mW\mC$ with all $n_p\in\bN$, define the \textit{weak composition monomial quasisymmetric function} (\textit{WCMQF})
\[M_\al=\sum_{1\leq i_1<\cdots<i_\ell}x_{i_1}^{n_1}\cdots x_{i_\ell}^{n_\ell},\]
lying in $\bZ[[x_1,x_2,\dots]]_\bN$, the ring of \textit{$\bN$-exponent formal power series}. Let $M_\emptyset=1$. Then $\{M_\al\}_{\al\in\mW\mC}$ forms a $\bZ$-basis of the ring of \textit{weak composition quasisymmetric functions}, defined in \cite{GTY} and denoted by WCQSym. Also denote QSym the subring of \textit{quasisymmetric functions} spanned by $\{M_\al\}_{\al\in\mC}$ as usual.
Let WCQSym$_n$ be the submodule of WCQSym spanned by $\{M_\al\}_{\al\in\mW\mC(n)}$, then WCQSym$=\bigoplus_{n\geq0}$WCQSym$_n$. The multiplication of $M_\al$'s in WCQSym is given by the quasi-shuffle product on the free abelian group $\bZ\mW\mC$, when WCQSym is identified with $\bZ\mW\mC$ as abelian groups. The \textit{quasi-shuffle} product $*$ on $\bZ\mW\mC$ is defined recursively as follows,
\[\al*\emptyset=\emptyset*\al=\al,\,\al*\be=(r_1,\al'*\be)+(s_1,\al*\be')+(r_1+s_1,\al'*\be')\]
for $\al=(r_1,\dots,r_k)=(r_1,\al'),\,\be=(s_1,\dots,s_l)=(s_1,\be')\in\mW\mC$.
It generalizes the \textit{shuffle} product $\shuffle$ on $\bZ\mW\mC$ recursively defined by
\[\al\shuffle\emptyset=\emptyset\shuffle\al=\al,\,\al\shuffle\be=(r_1,\al'\shuffle\be)+(s_1,\al\shuffle\be').\]
For convenience, we define a $\bZ$-bilinear pair $\lan\cdot,\cdot\ran$ on $\bZ\mW\mC$ satisfying $\lan \al,\be\ran=\de_{\al\be}$. Then $M_\al\cdot M_\be=\sum\limits_{\ga\in\mW\mC}\lan\ga,\al*\be\ran M_\ga$, and WCQSym is an $\bN$-graded ring with respect to the weight.

On the other hand,
for any $\al=(0^{j_1},s_1,\dots,0^{j_k},s_k,0^{j_{k+1}})\in\mW\mC$ with all $j_p\in\bN,\,s_q\in\bP$, define the \textit{weak composition fundamental quasisymmetric function} (\textit{WCFQF})
\[F_\al=\sum_{1\leq i_1\leq\cdots\leq i_{a_k+j_{k+1}}\atop j\in D(\al)\Rightarrow i_{j}<i_{j+1}}
x_{i_1}^0\dots x_{i_{j_1}}^0x_{i_{j_1+1}}\cdots x_{i_{a_1}}\cdots
x_{i_{a_{k-1}+1}}^0\dots x_{i_{a_{k-1}+j_k}}^0x_{i_{a_{k-1}+j_k+1}}\cdots x_{i_{a_k}}
x_{i_{a_k+1}}^0\dots x_{i_{a_k+j_{k+1}}}^0\]
where $D(\al):=\{a_1,\dots,a_k\}$.  Let also $F_\emptyset=1$.
In \cite[Prop. 3.1]{GTY}, the authors prove the expansion formula
\beq\label{FM}
F_\al=\sum_{\be\leq\al}c_{\al\be}M_\be,\, c_{\al\be}={i_1\choose j_1}\cdots{i_k \choose j_k}{i_{k+1}-1\choose j_{k+1}-1}
\eeq
for $\al=(0^{i_1},\al_1,\dots,0^{i_k},\al_k,0^{i_{k+1}}),\,\be=(0^{j_1},
\be_1,\dots,0^{j_k},\be_k,0^{j_{k+1}})\in\mW\mC(n)$. Here we use the convention that ${-1 \choose -1}=1$ and ${i \choose j}=0$ if $i\in\bN,\,j\in\bZ$ such that $j<0$ or $i<j$, also for the rest of the paper.
Thus $\{F_\al\}_{\al\in\mW\mC(n)}$ is another $\bZ$-basis of WCQSym$_n$, and by \cite[Prop. 3.2]{GTY},
\beq\label{MF}
M_\al=\sum_{\be\leq\al}(-1)^{\ell(\be)-\ell(\al)}c_{\al\be}F_\be.
\eeq


\subsection{Peak quasisymmetric functions}
For completeness, we briefly recall the ring of \textit{peak quasisymmetric functions} defined in \cite{Ste} and denoted by PQSym. It is a graded subring of QSym. We say $P$ is a \textit{peak subset} of $[n]$, if $P\subseteq[2,n-1]$ and $i\in P\Rightarrow i-1\notin P$. Given $\al\in\mC(n)$, let
\[P(\al):=\{i\in D(\al)\cap[2,n-1]\,|\,i-1\notin D(\al)\}\]
be its associated peak subset of $[n]$. Let $\mP_n$ be the set consisting of all peak subsets of $[n]$.

Recall that Stembridge's \textit{peak functions} in PQSym can be defined by
\beq\label{ka}K_P=2^{|P|+1}\sum_{\al\vDash n\atop P\subseteq D(\al)\triangle(D(\al)+1)}F_\al,\,P\in\mP_n,\eeq
where $D\triangle(D+1)=D\backslash(D+1)\cup (D+1)\backslash D$ for any $D\subseteq[n-1]$ and $D+1:=\{x+1:x\in D\}$. Then $\{K_P\}_{P\in\mP_n}$ forms a $\bZ$-basis of PQSym$_n$ and there also exists a surjective ring homomorphism
\[\te:\mb{QSym}\rw\mb{PQSym},\quad F_\al\mapsto K_{P(\al)},\]
called the \textit{descent-to-peak map}. Note that
\beq\label{ka}
K_P=\sum_{\al\vDash n\atop P\subseteq D(\al)\cup(D(\al)+1)}2^{\ell(\al)}M_\al,\,P\in\mP_n.
\eeq
By convenience, we also write $K_\al:=K_{P(\al)}$ for any $\al\in\mC(n)$.

%

\section{Enriched $P$-partitions and weak peak quasisymmetric functions}

In \cite{GTY}, the authors generalize the construction of fundamental quasisymmetric functions by use of $P$-partitions to define the WCFQFs. In \cite{Ste}, Stembridge realizes peak quasisymmetric functions via the theory of enriched $P$-partitions. Referring to these papers, we apply the concept of enriched $P$-partitions to construct other special kinds of weak composition quasisymmetric functions here.

\subsection{Enriched $P$-partitions}
First we recall Stembridge's theory of enriched $P$-partitions in \cite{Ste}.
Now totally order $\bZ'$ so that
\[-1\prec 1\prec -2\prec 2\prec -3\prec 3\prec\cdots.\]
A partial ordering $(P,<_P)$ of a finite subset $X$ of $\bP$ is called a \textit{labeled poset} on $X$, and $P=X$ also has the natural order $<$ inheriting from $\bP$. An \textit{enriched $P$-partition} is
a map $f:P\rw\bZ'$ such that for all $x<_P y$ in $P$, we have
\begin{align*}
\mb{(i)}&\quad f(x)\preceq f(y),\\
\mb{(ii)}&\quad f(x)=f(y)>0\mb{ implies }x<y,\\
\mb{(iii)}&\quad f(x)=f(y)<0\mb{ implies }x>y.
\end{align*}
Let $\mE(P)$ denote the set of enriched $P$-partitions. In particular, if $P$ is a chain $\{w_1<_P\cdots<_Pw_n\}$, it can be identified with a one-line word $w=w_1\cdots w_n$ of positive integers.
Define its descent set $D(P):=\{i\in[n-1]\,|\,w_i>w_{i+1}\}$. Then
\[\begin{split}
\mE(P)=\{f:P\rw\bZ'\,|\,&f(w_1)\preceq\cdots \preceq f(w_n),\\
&f(w_i)=f(w_{i+1})>0\Rightarrow i\notin D(P),\\
&f(w_i)=f(w_{i+1})<0\Rightarrow i\in D(P)\}.
\end{split}\]
Given a subset $S$ of $P$, we define the \textit{weight} of $f\in\mE(P)$ to be the monomial
\[w(f,S):=\prod_{v\in P}x_{|f(v)|}^{\de_v},\mb{ where }\de_v=
\begin{cases}
0,& v\notin S,\\
1,& v\in S,
\end{cases}\]
and define the \textit{weight enumerator} for $P$ as the generating function
\[\La(P,S):=\sum_{f\in\mE(P)}w(f,S)\in\bZ[[x_1,x_2,\dots]]_\bN.\]
Let $\mL(P)$ denote the set of linear extensions of $P$ compatible with the partial order $<_P$. By \cite[Lemma 2.1]{Ste} we have  $\mE(P)=\bigcup\limits_{u\in\mL(P)}^{.}\mE(u)$, thus
\beq\label{lin}\La(P,S)=\sum_{u\in\mL(P)}\La(u,S).\eeq

Given two disjoint posets $P,Q$, the \textit{direct sum} of $P$ and $Q$ is the poset $P+Q$ on $P\cup Q$ with the partial order only retaining those of $P$ and $Q$ respectively, and the \textit{ordinal sum} of $P$ and $Q$ is the poset $P\oplus Q$ on $P\cup Q$ further demanding $p<q$ for all $p\in P,q\in Q$. When $P,Q$ are two disjoint labeled posets with $S\subseteq P,T\subseteq Q$, then $P+Q,P\oplus Q$ are labeled posets with a subset $S\cup T$.
\begin{lem}
Given two disjoint labeled posets $P,Q$ with $S\subseteq P,T\subseteq Q$,
\beq\label{mul}\La(P,S)\La(Q,T)
=\sum_{u\in\mL(P)\shuffle\mL(Q)}\La(u,S\cup T)\eeq
where $\mL(P)\shuffle\mL(Q):=\{u=x\shuffle y\,|\,x\in\mL(P),y\in\mL(Q)\}$ with $\shuffle$ as the shuffle operation of two chains.
\end{lem}
\begin{proof}
First by definition, we have
\[\La(P,S)\La(Q,T)=\sum_{f\in\mE(P)\atop g\in\mE(Q)}
w(f,S)w(g,T)=\sum_{h\in\mE(P+Q)}w(h,S\cup T)=\La(P+Q,S\cup T).\]
Note that $\mL(P+Q)=\mL(P)\shuffle\mL(Q)$, hence by \eqref{lin},
\[\La(P,S)\La(Q,T)
=\sum_{u\in\mL(P+Q)}\La(u,S\cup T)
=\sum_{u\in\mL(P)\shuffle\mL(Q)}\La(u,S\cup T)\qedhere\]
\end{proof}

Now given a weak composition $\al=(0^{i_1},s_1,\dots,0^{i_k},s_k,0^{i_{k+1}})$ with all $i_p\in\bN,\,s_q\in\bP$,
one can define a labeled poset $P_\al$ on $[\|\al\|]$ as follows. Let a chain $P_\al:=C_1\oplus P_1\oplus\cdots\oplus C_k\oplus P_k\oplus C_{k+1}$, where $C_p$ and $P_p$ are chains with $i_p$ and $s_p$
elements respectively. Make $P_\al$ into a labeled poset by numbering first the elements
of $C_1,C_2,\dots,C_{k+1}$ and then the elements of $P_k, P_{k-1},\dots, P_1$, such that the labeled order is compatible with the order of $C_p$ and $P_p$. For example, if $\al=(0,1,3,0^2,2,0)$, then
$C_2=\emptyset$ and $P_\al=\{1\}\oplus\{10\}\oplus\emptyset\oplus\{7,8,9\}\oplus\{2,3\}\oplus\{5,6\}\oplus\{4\}$.
In the rest of the paper, if without further illustration, we always choose the subset $S=\cup_pP_p$ for the weight of any $f\in\mE(P_\al)$, and abbreviate $\La(P_\al,S)$ as $\La(P_\al)$.

Moreover, we introduce the following notation for convenience.
For any labeled poset $P$ on $[n]$ with $S\subseteq P$, when $P$ has a one-line word presentation as a chain, we always highlight $P\setminus S$ by putting lines over the labels in $P\setminus S$. For instance, if $\al=(0,1,2,0^2,2,0)$, we write $P_\al$ as $\ol{1}978\ol{23}56\ol{4}$, thus $\La(P_\al)$ as $\La(\ol{1}978\ol{23}56\ol{4})$.

\begin{rem}
In \cite{GTY}, the authors describe the WCFQFs $F_\al$ via $P_\al$-partitions. Under the same assumption as enriched $P$-partitions, a \textit{$P$-partition} is
a map $f:P\rw\bP$ such that for all $x<_P y$ in $P$, we have
\begin{align*}
\mb{(i)}&\quad f(x)\leq f(y),\\
\mb{(ii)}&\quad f(x)<f(y)\mb{ if }x>y.
\end{align*}
Let $\mA(P)$ denote the set of $P$-partitions. One can analogously define the generating function
\[\Ga(P,S):=\sum_{f\in\mA(P)}w(f,S)\in\bZ[[x_1,x_2,\dots]]_\bN.\]
If we abbreviate $\Ga(P_\al,S)$ as $\Ga(P_\al)$ under the default choice that $S=\cup_pP_p$, then $F_\al=\Ga(P_\al)$.
\end{rem}

\subsection{Weak peak quasisymmetric functions}
It is clear that $\La(P_\al)$ is also a homogeneous weak composition quasisymmmetric function. We called these functions $\La(P_\al)$ \textit{weak peak fundamental quasisymmetric functions} (\textit{WPFQF}).
To describe the expansions of $\La(P_\al)$ in terms of WCMQFs and also WCFQFs, we need the following indispensable notion. Given a weak composition $\al$, let
\[P(\al):=\{i\in D(\al)\cap[2,\|\al\|-1]\,|\,i-1\notin D(\al)\}\]
be its associated \textit{peak set}. Any weak composition with an empty peak set has the form
\[(\overbrace{1,\dots,1}^i,\overbrace{0,\dots,0}^j)\mb{ or }(\overbrace{1,\dots,1}^i,\overbrace{0,\dots,0}^j,k),i,j\in\bN,k\in\bP.\]


Given two weak compositions $\al,\be$, we say $\al$ can \textit{mutate to} $\be$, denoted by $\be\unlhd\al$, if one can split positive parts of $\al$ into pieces if necessary to obtain a decomposition of weak compositions
$(\al_1,\dots,\al_\ell)$, where all $\al_q\in\mW\mC$ have empty peak sets such that
splitting pieces from one positive part of $\al$ lie in different $\al_q$ and
$(|\al_1|,\dots,|\al_\ell|)=\be$ with $\ell=\ell(\be)$. Then we say this decomposition of $\al$ satisfies \textit{condition} ($*$) for $\be$. Generally such decomposition is not unique. Denote $n_{\al\be}$ the number of all decompositions of $\al$
satisfying condition ($*$) for $\be$. In particular, $n_{\al\al}=1$. Let $\emptyset$ be the only mutation of itself with $n_{\emptyset\al}=\de_{\al,\emptyset}$ for convenience. If $\al,\be$ are compositions, $n_{\al\be}\leq1$.


\begin{exam}\label{ex}
Let $\al=(0^2,4,0^2,2,0^2,1,0,1)$ and $\be=(2,1,3,0,1,1)$, then $\be\unlhd\al$ as $\al$ has three decompositions,
\begin{align*}
&(\al_1(0^2,2),\al_2(1),\al_3(1,0^2,2),\al_4(0^2),\al_5(1),\al_6(0,1)),\\
&(\al_1(0^2,2),\al_2(1),\al_3(1,0^2,2),\al_4(0^2),\al_5(1,0),\al_6(1)),\\
&(\al_1(0^2,2),\al_2(1),\al_3(1,0^2,2),\al_4(0),\al_5(0,1),\al_6(0,1)),
\end{align*}
satisfying condition ($*$) for $\be$, and $n_{\al\be}=3$.
\end{exam}

Now we are in the position to give the first characterization of WPFQFs.
\begin{prop}\label{p1}
For any weak composition $\al$, we have
\beq\label{PM}
\La(P_\al)=\sum_{\be\unlhd\al}n_{\al\be}2^{\ell(\be)}M_\be.
\eeq
\end{prop}
\begin{proof}
Fix $\al=(0^{i_1},s_1,\dots,0^{i_k},s_k,0^{i_{k+1}})$ with all $i_l\in\bN$ and $s_p\in\bP$. For any weak composition $\be=(0^{j_1},t_1,\dots,0^{j_r},t_r,0^{j_{r+1}})$  with all $j_l\in\bN$ and $t_q\in\bP$, let $l_q:=j_1+\cdots+j_q,\,q=1,\dots,r+1$.
First note that the coefficient of $M_\be$ in $\La(P_\al)$ is also that of the monomial
\beq\label{mon}
x_1^0\cdots x_{l_1}^0x_{l_1+1}^{t_1}x_{l_1+2}^0\cdots x_{l_2+1}^0x_{l_2+2}^{t_2}\cdots x_{l_{r-1}+r}^0\cdots x_{l_r+r-1}^0 x_{l_r+r}^{t_r}x_{l_r+r+1}^0\cdots x_{l_{r+1}+r}^0 \eeq
in $\La(P_\al)$. By the definition of enriched $P_\al$-partitions, any factor $x_l^0$ in such monomial must come from

(0) the weight on an interval in one $C_p$. Namely, $\prod\limits_{i=1}^n x_{|f(w_i)|}^0=x_l^0,\,n\in\bP$,  where $\{w_1,\dots,w_n\}$ forms an interval in $C_p$, and $(f(w_1),\dots,f(w_n))=(\pm l,\overbrace{l,\dots,l}^{n-1})$.

Meanwhile, any factor $x_{l_q+q}^{t_q}$ must belong to one of the following three cases:

(1) the weight on an interval in one $C_p\oplus P_p$, $\prod\limits_{i=1}^m x_{|f(u_i)|}^0
\prod\limits_{j=1}^{t_q} x_{|f(v_j)|}=x_{l_q+q}^{t_q},\,m\in\bN$, where $\{u_1,\dots,u_m\}\oplus\{v_1,\dots,v_{t_q}\}$ forms an interval in $C_p\oplus P_p$, and \[(f(u_1),\dots,f(u_m),f(v_1),\dots,f(v_{t_q}))=(\pm(l_q+q),\overbrace{l_q+q,\dots,l_q+q}^{m+t_q-1}).\]

(2) the weight on an interval in one $P_{p-t_q}\oplus\cdots\oplus P_{p-1}\oplus C_p$ with $C_{p-t_q+1},\dots,C_{p-1}=\emptyset$ and $P_{p-t_q+1},\dots,P_{p-1}$ are singleton sets, $\prod\limits_{j=1}^{t_q} x_{|f(v_j)|}\prod\limits_{i=1}^m x_{|f(u_i)|}^0=x_{l_q+q}^{t_q},\,m\in\bN$, where $\{v_1\}\oplus\cdots\oplus\{v_{t_q}\}\oplus\{u_1,\dots,u_m\}$ forms an interval in $P_{p-t_q}\oplus\cdots\oplus P_{p-1}\oplus C_p$, and
\[\begin{split}
(f(v_1)&,\dots, f(v_{t_q}),f(u_1),\dots,f(u_m))\\
&=(\overbrace{-(l_q+q),\dots,-(l_q+q)}^{t_q},
\pm(l_q+q),\overbrace{l_q+q,\dots,l_q+q}^{m-1}).
\end{split}\]

(3) the weight on an interval in one $P_{p-n}\oplus\cdots\oplus P_{p-1}\oplus C_p\oplus P_p$ with $C_{p-n+1},\dots,C_{p-1}=\emptyset$ and $P_{p-n+1},\dots,P_{p-1}$ are singleton sets, $\prod\limits_{j=1}^n x_{|f(w_j)|}\prod\limits_{i=1}^m x_{|f(u_i)|}^0\prod\limits_{j=1}^{t_q-n} x_{|f(v_j)|}=x_{l_q+q}^{t_q},\,m\in\bN,n\in[t_q-1]$,  where $\{w_1\}\oplus\cdots\oplus\{w_n\}\oplus\{u_1,\dots,u_m\}\oplus\{v_1,\dots,v_{t_q-n}\}$ forms an interval in $P_{p-n}\oplus\cdots\oplus P_{p-1}\oplus C_p\oplus P_p$, and
\[\begin{split}
(f(w_1)&,\dots, f(w_n),f(u_1),\dots,f(u_m),f(v_1),\dots,f(v_{t_q-n}))\\
&=(\overbrace{-(l_q+q),\dots,-(l_q+q)}^n,
\pm(l_q+q),\overbrace{l_q+q,\dots,l_q+q}^{m+t_q-n-1}).
\end{split}\]

Note that intervals in $P_\al$ mentioned in cases (0)--(3) clearly represent weak compositions with empty peak sets. By the analysis above, if the coefficient of $M_\be$ in $\La(P_\al)$ is nonzero, $\al$ should have decompositions of weak compositions satisfying condition ($*$) for $\be$. That is $\be\unlhd\al$.

To obtain the expansion of $\La(P_\al)$ in terms of WCMQFs, we first note that monomial \eqref{mon} in $\La(P_\al)$ can be obtained by sequentially piecing together $\ell(\be)$ weights from cases (0)--(3) above when $\be\unlhd\al$. For each of these weights, one should select a sign in $\pm$ to fix the value of $f$. Also such approach is not necessarily unique. By definition the number of ways merging weights to derive monomial \eqref{mon} is just $n_{\al\be}$. Hence we obtain the expansion formula \eqref{PM}.
\end{proof}

Next we describe the number $n_{\al\be}$ more explicitly by use of the notation in Prop. \ref{p1}. Indeed from another perspective, the number $n_{\al\be}$ counts how many kinds of distributions of weights on $C_p,\,p=1,\dots,k+1$, into weights in cases (0)--(3) contributing to monomial \eqref{mon}. There mainly appear four cases.

(i) If the weight on $C_p$ splits into parts in (0) and (1) (resp. (2) and (0)) successively for the weight $x_{l_{q-1}+q}^0\cdots x_{l_q+q-1}^0 x_{l_q+q}^{t_q}$ (resp. $x_{l_{q-1}+q-1}^{t_{q-1}}x_{l_{q-1}+q}^0\cdots x_{l_q+q-1}^0$), then it has $i_p \choose j_q$ choices.

(ii) If the weight on $C_p$ splits into parts in (2), (0) and (1) successively for the weight $x_{l_{q-1}+q-1}^{t_{q-1}}x_{l_{q-1}+q}^0\cdots x_{l_q+q-1}^0 x_{l_q+q}^{t_q}$, then it has $i_p+1\choose j_q+1$ choices.

(iii) If the weight on $C_p$ distributes to that in (0) for the weight $x_{l_{q-1}+q}^0\cdots x_{l_q+q-1}^0$, then it has $i_p-1\choose j_q-1$ choices.

(iv) If the weight on $C_p$ distributes to that in (3) for the weight $x_{l_q+q}^{t_q}$, the choice is unique.

However, there exists an exceptional case. If $P_{p-n+1},\dots,P_{p}$ are singleton sets and the distribution of weight on $C_{p-n+1}\oplus P_{p-n+1}\oplus\cdots\oplus C_p\oplus P_p\oplus C_{p+1}$ is for the weight
\[x_{l_{q-n}+q-n+1}^0\cdots x_{l_{q-n+1}+q-n}^0 x_{l_{q-n+1}+q-n+1}\cdots x_{l_{q-1}+q}^0\cdots x_{l_q+q-1}^0 x_{l_q+q}x_{l_q+q+1}^0\cdots x_{l_{q+1}+q}^0,\]
it has
\beq\label{11}
\sum_{\ep_{q-n+1},\dots,\ep_q\in\{0,1\}}{i_{p-n+1}-1 \choose j_{q-n+1}-\ep_{q-n+1}}
{i_{p-n+2}-1+\ep_{q-n+1} \choose j_{q-n+2}-\ep_{q-n+2}+\ep_{q-n+1}}
\cdots{i_p-1+\ep_{q-1} \choose j_q-\ep_q+\ep_{q-1}}
{i_{p+1}-1+\ep_q\choose j_{q+1}-1+\ep_q}
\eeq
choices. In fact, for any $i=1,\dots,n$, when $\ep_{q-i+1}=0$ it means that if $C_{p-i+1}\neq\emptyset$, the last element in $C_{p-i+1}$ contributes to $x_{l_{q-i+1}+q-i+1}^0$, which combines into the $x_{l_{q-i+1}+q-i+1}$ from $P_{p-i+1}$, so the first element in $C_{p-i+2}$ can not take such weight by the constraint of enriched $P_\al$-partitions. On the other hand, $\ep_{q-i+1}=1$ means that the last element in $C_{p-i+1}$ contributes to $x_{l_{q-i+1}+q-i}^0$ instead, and the weight from the first element in $C_{p-i+2}$ can be $x_{l_{q-i+1}+q-i+1}^0$. In summary, as $C_{p-i+1}$ should provide $x_{l_{q-i}+q-i+1}^0\cdots x_{l_{q-i+1}+q-i}^0$ and $P_{p-i+1}$ for $x_{l_{q-i+1}+q-i+1}$, we get the desired number of choices.

Note that formula \eqref{11} can be extended to involve the case (iv), namely some $P_{p-r-1}\oplus C_{p-r}\oplus P_{p-r}$, $r=0,\dots,n-2$, merge to provide one $x_l^2$. In this situation, we set $j_{q-r}=-1$ in \eqref{11} by convenience, together with $\ep_{q-r-1}=1$ and $\ep_{q-r}=0$, to imply such merging. It gives $i_{q-r}-1+\ep_{q-r-1}\choose j_{q-r}-\ep_{q-r}+\ep_{q-r-1}$=1, thus the accurate number $n_{\al\be}$. For example, $\al=(0,1,0^2,1,0,1,0,2)\unlhd \be=(0,2,1,2)$. Write $\be=(0,1,0^{-1},1,0^0,1,0^0,2)$, hence $n_{\al\be}={1-1\choose1-1}{2-1+1\choose-1-0+1}
{1-1+0\choose0-0+0}^2=1$.


\begin{exam}
For $\al=(0^{i_1},1,0^{i_2},2,0^{i_3},1)$, there exist the following weak compositions $\unlhd\al$,
\[\begin{split}
&\al_{j_1j_2j_3}=(0^{j_1},1,0^{j_2},2,0^{j_3},1),\,
\be_{j_1j_2j_3}=(0^{j_1},1,0^{j_2},1^2,0^{j_3},1),\\
&\ga_{j_1j_2}=(0^{j_1},1,0^{j_2},1,2),\,\eta_{j_1j_3}=(0^{j_1},2,1,0^{j_3},1),\\
&\xi_{j_1j_3}=(0^{j_1},3,0^{j_3},1),\,\zeta_{j_1}=(0^{j_1},2^2)
\end{split}\]
with $j_p\in\{0,1,\dots,i_p\}$ for $p=1,2,3$. Then
\[\begin{split}
\La(P_\al)&=\sum_{j_1,j_2,j_3}\lb{i_1-1\choose j_1-1}{i_2+1\choose j_2+1}+{i_1-1\choose j_1}{i_2\choose j_2}\rb\\
&\lb{i_3\choose j_3}2^{\ell(\al_{j_1j_2j_3})}M_{\al_{j_1j_2j_3}}+{i_3+1\choose j_3+1}2^{\ell(\be_{j_1j_2j_3})}M_{\be_{j_1j_2j_3}}+
2^{\ell(\ga_{j_1j_2})}M_{\ga_{j_1j_2}}\rb\\
&+{i_1-1\choose j_1-1}\lb{i_3+1\choose j_3+1}2^{\ell(\eta_{j_1j_3})}M_{\eta_{j_1j_3}}
+{i_3\choose j_3}2^{\ell(\xi_{j_1j_3})}M_{\xi_{j_1j_3}}
+2^{\ell(\zeta_{j_1})}M_{\zeta_{j_1}}\rb.
\end{split}\]

\end{exam}

\begin{lem}\label{eq}
Given two weak compositions $\al,\be$, $\La(P_\al)=\La(P_\be)$ if and only if
\[\al=(0^{i_1},\mu_1,\dots,0^{i_k},\mu_k,0^{i_{k+1}}),\,\be=(0^{i_1},\nu_1,\dots,0^{i_k},\nu_k,0^{i_{k+1}}),\]
where $i_1,\dots,i_{k+1}\in\bN$, $\mu_q,\nu_q$ are two compositions of the same weight for $q=1,\dots,k$ and $P(\al)=P(\be)$.
\end{lem}
\begin{proof}
The given condition above is equivalent to say that $\al$ and $\be$ only differ by some splitting of positive parts preserving that $P(\al)=P(\be)$. It certainly ensures that $\ga\unlhd\al$ if and only if $\ga\unlhd\be$ and $n_{\al\ga}=n_{\be\ga}$, thus $\La(P_\al)=\La(P_\be)$ by formula \eqref{PM}.
Conversely, since $M_\ga$'s are linearly independent, $\La(P_\al)=\La(P_\be)$ forces that
$\ga\unlhd\al$ if and only if $\ga\unlhd\be$. In particular,
$\al\unlhd\be$ and $\be\unlhd\al$. Hence, every interval of zeros in $\be$ appears as a contraction of exactly one $0^{i_p}$ in $\al$, and vice versa. It implies that $\al,\be$ have the same intervals $0^{i_p}$ of zeros as stated above, and $\al,\be$ differ by some splitting of positive parts, thus $|\mu_q|=|\nu_q|$ for $q=1,\dots,k$. Moreover, $\La(P_\al)=\La(P_\be)$ immediately implies that there exists a bijection between
decompositions $(\al_1,\dots,\al_\ell)$ of $\al$ into weak compositions with empty peak sets
and those decompositions $(\be_1,\dots,\be_\ell)$ of $\be$ such that
\[(|\al_1|,\dots,|\al_\ell|)=(|\be_1|,\dots,|\be_\ell|)\unlhd\al,\be.\]
It guarantees that $P(\al)=P(\be)$.
\end{proof}
For example, if $\al=(0^2,3,0^2,2.,0^2,1,1)$ and $\be=(0^2,1,2,0^2,2,0^2,2)$, $\La(P_\al)=\La(P_\be)$. Here it is necessary to require that each pair of $\al_q,\be_q$ are of the same weight. Indeed, there exist weak compositions with the same intervals of zeros and peak set but different functions, such as $(0,2,1,0^2,1,1)$ and $(0,2,0^2,2,1)$ both with peak set $\{3,7\}$.

Given two weak compositions $\al,\be$ such that $\be\unlhd\al$, we write $\be\Vdash\al$, if no positive part of $\al$ splits into pieces with the last one becoming a part 1 of $\be$ followed by a positive part. For instance, if $\al=(0^2,5,0^2,2),\,\be=(0,2^2,1,0,2)$ and $\ga=(0,2^2,1,2)$, we have $\be\Vdash\al$ and $\ga\unlhd\al$ but $\ga\nVdash\al$, as one has to split the $5$ in $\al$ into $2^2,1$ with the last $1$ for the $1$ in $\be$ left to a $0$, but for the $1$ in $\ga$ left to $2$.

Next we consider the expansion of $\La(P_\al)$ in terms of WCFQFs. Unlike the classical case \eqref{ka}, such expansion is no longer positive in general. For instance, $\La(P_{01})=4M_{01}+2M_1=4F_{01}-2F_1$ as $F_1=M_1$ and $F_{01}=M_1+M_{01}$.
\begin{prop}\label{p2}
For any weak composition $\al$, we have
\beq\label{PF}
\La(P_\al)=\sum_{\be\Vdash\al}t_{\al\be}F_\be
\eeq
where $t_{\al\be}=\sum\limits_{\be\leq\ga\unlhd\al}n_{\al\ga}c_{\ga\be}(-1)^{\ell(\be)-\ell(\ga)}2^{\ell(\ga)}$.
\end{prop}
\begin{proof}
First by definition, it is clear that given two $\al,\be\in\mW\mC$,
$\be\unlhd\al$ if there exists $\ga\in\mW\mC$ such that $\be\leq\ga\unlhd\al$.
Then by formulas \eqref{MF} and \eqref{PM},
\[\begin{split}
\La(P_\al)&=\sum_{\ga\unlhd\al}n_{\al\ga}2^{\ell(\ga)}M_\ga
=\sum_{\ga\unlhd\al}n_{\al\ga}2^{\ell(\ga)}\lb
\sum_{\be\leq\ga}(-1)^{\ell(\be)-\ell(\ga)}c_{\ga\be}F_\be\rb\\
&=\sum_{\be\unlhd\al}\lb\sum_{\be\leq\ga\unlhd\al}n_{\al\ga}c_{\ga\be}(-1)^{\ell(\be)-\ell(\ga)}
2^{\ell(\ga)}\rb F_\be\triangleq\sum_{\be\unlhd\al}t_{\al\be}F_\be.
\end{split}\]

Now we only need to show that the coefficient $t_{\al\be}$ vanishes if $\be\unlhd\al$ but $\be\nVdash\al$, i.e. there exists a splitting of a positive part in $\al$ with the last piece for a part 1 in $\be$ followed by a positive one. By the definition of $t_{\al\be}$, it is clearly enough to deal with the special case
when $\al=(m,0^r,n)$ and $\be=(m-1,1,n)$ with $m\geq2,n\geq1,r\geq0$, as which it is nearly the same to check the other case when $\al=(1,0^p,m,0^q,n)$ and $\be=(m,1,n)$ with $m\geq2,n\geq1,p,q\geq0$.

For $\al=(m,0^r,n)$ and $\be=(m-1,1,n)$, all $\ga$'s satisfying $\be\leq\ga\unlhd\al$ are as follows,
\[(m,0^j,n),(m-1,1,0^k,n),(m-1,n+1),\,0\leq j,k\leq r.\]
Hence we have
\[\begin{split}
t_{\al\be}&=\sum_{j=0}^r{r\choose j}(-1)^{j-1}2^{j+2}+
\sum_{k=0}^r{r+1\choose k+1}(-1)^k2^{k+3}+(-1)^{3-2}2^2\\
&=\sum_{j=0}^r{r\choose j}(-1)^{j-1}2^{j+2}+
\sum_{k=0}^r\lb{r\choose k}+{r\choose k+1}\rb(-1)^k2^{k+3}-4\\
&=\sum_{j=0}^r{r\choose j}(-1)^j2^{j+2}+
\sum_{k=1}^r{r\choose k}(-1)^{k-1}2^{k+2}-4=2^2-4=0.
\end{split}\]
In summary, we finally prove that $\La(P_\al)=\sum\limits_{\be\Vdash\al}t_{\al\be}F_\be$.
\end{proof}
Though the coefficients $t_{\al\be}$ are not always positive, we still wonder a cancellation-free description for them, which is mysterious to us.


\begin{exam}
For $\al=(2,0^r,1)$, we have
\begin{align*}
&(2,0^i,1)\leq(2,0^j,1)\unlhd(2,0^r,1),\\
&(1^3),(1,2)\leq(1,2)\unlhd(2,0^r,1),\\
&(1^2,0^i,1)\leq(2,0^j,1),(1^2,0^k,1)\unlhd(2,0^r,1),\,0\leq i\leq j,k\leq r,
\end{align*}
thus by formula \eqref{PF},
\[\begin{split}
\La(P_{20^r1})&=\sum_{i=0}^r\lb\sum_{j=i}^r{r\choose j}{j\choose i}(-1)^{j-i}2^{j+2}\rb F_{20^i1}
+4F_{12}\\
&+\sum_{i=0}^r\lb\sum_{j=i}^r{r\choose j}{j\choose i}(-1)^{j-i-1}2^{j+2}+
\sum_{k=i}^r{r+1\choose k+1}{k\choose i}(-1)^{k-i}2^{k+3}-4\de_{i0}\rb F_{1^20^i1}.
\end{split}\]
In particular, $\La(P_{20^31})=32F_{20^31}-48F_{20^21}+24F_{201}-4F_{21}+32F_{1^20^31}-16F_{1^20^21}+8F_{1^201}+4F_{12}$.
\end{exam}

\section{The algebra of weak peak quasisymmetric functions}
Let WPQSym$_n$ be the submodule of WCQSym$_n$ spanned by $\{\La(P_\al)\}_{\al\in\mW\mC(n)}$.
Denote WPQSym$=\bigoplus_{n\geq0}$WPQSym$_n$.
When $\al$ is a composition, $\be\unlhd\al$ (resp. $\ga\Vdash\al$) if and only if $\be$ (resp. $\ga$) is a composition satisfying $P(\al)\subseteq D(\be)\cup(D(\be)+1)$ (resp. $P(\al)\subseteq D(\ga)\triangle(D(\ga)+1)$), where $\triangle$ denotes the symmetric difference. As a result, $\La(P_\al)$ recovers Stembridge's \textit{peak functions} $K_{P(\al)}$ defined in \cite{Ste}. From now on, we let $K_\al:=\La(P_\al)$ for any $\al\in\mW\mC$ with $K_\emptyset=1$, generalizing such classical peak functions.

There is another difference from the classical case, namely, the multiplication of two WCFQFs is not always a positive expansion in terms of WCFQFs. In fact, for two weak compositions $\al,\be$,
\beq\label{muf}F_\al\cdot F_\be=\Ga(P_\al)\cdot\Ga(P_\be)=\sum_{u\in(P_\al\sim\be)\shuffle(P_\be\smile\al)}\Ga(u),\eeq
where we denote $P_\al\sim\be$ as a copy of $P_\al$ with all labels in $P_q$'s shifting $\ell_0(\be)$, and $P_\be\smile\al$ as a copy of $P_\be$ with all labels in $C_p$'s (resp. $P_q$'s) shifting $\ell_0(\al)$ (resp. $\|\al\|$). For any  $u\in(P_\al\sim\be)\shuffle(P_\be\smile\al)$, $\Ga(u)$ can be written as a $\bZ$-linear combination of the WPFQFs but not always positive.

\begin{rem}
For any composition $\al$ of $n$, the function $\Ga(P_\al,\emptyset)$ has the following expression,
\beq\label{gm}
\Ga(P_\al,\emptyset)=\sum_{j=\ell}^n
{n-\ell\choose j-\ell}M_{0^j},
\eeq
where $\ell=\ell(\al)$. Moreover, $\Ga(P_\al,\emptyset)$ can be $\bZ$-linearly expanded by $F_{0^j}$'s as follows.
\beq\label{gf}
\Ga(P_\al,\emptyset)=\sum_{j=0}^{\ell-1}(-1)^j{\ell-1 \choose j}F_{0^{n-j}}.
\eeq
\end{rem}
As a matter of fact, we can use \eqref{muf} and \eqref{gf} to obtain the expansion of $F_\al\cdot F_\be$ in terms of WCFQFs. For $\al=(1,0),\,\be=(0^2,1)$, we have $P_\al\sim\be=4\ol{1},\,P_\be\smile\al=\ol{23}5$, and
\begin{align*}
&F_{10}\cdot F_{0^21}=\sum_{u\in4\ol{1}\shuffle\ol{23}5}\Ga(u)=
\Ga(4\ol{123}5)+\Ga(4\ol{213}5)+\Ga(4\ol{231}5)+\Ga(4\ol{23}5\ol{1})\\
&\quad+\Ga(\ol{2}4\ol{13}5)+\Ga(\ol{2}4\ol{31}5)+\Ga(\ol{2}4\ol{3}5\ol{1})+\Ga(\ol{23}4\ol{1}5)
+\Ga(\ol{23}45\ol{1})+\Ga(\ol{23}54\ol{1})\\
&\quad=3F_{10^31}-2F_{10^21}
+F_{10^210}+2F_{010^21}-F_{0101}+F_{01010}
+F_{0^2101}+F_{0^220}+F_{0^21^20}.
\end{align*}

In particular, WCQSym has a subalgebra spanned by $F_{0^r},\,\forall r\in\bN$, with $F_{0^0}:=F_\emptyset=1$, and we denote it by WCQSym$^0$.
\begin{lem}\label{mufz}
Given $m,n\in\bP$, we have
\[F_{0^m}\cdot F_{0^n}=\sum_{j=0}^m(-1)^j{m\choose j}{m+n-j\choose m}F_{0^{m+n-j}}\]
\end{lem}
\begin{proof}
By symmetry, we can suppose that $m\leq n$. By \eqref{muf}, we first get that
\[F_{0^m}\cdot F_{0^n}=\Ga([m],\emptyset)\cdot\Ga([n],\emptyset)=\sum_{u\in[m]\shuffle[m+1,m+n]}\Ga(u,\emptyset).\]
Note that any $u\in[m]\shuffle[m+1,m+n]$ is isomorphic to one $P_\al$, as
the descents of $u$ determine a composition $\al$ of $m+n$. Then by \eqref{gf}, we need to count how many $u$'s correspond to compositions of length $\ell$ for any given $\ell\in[m+1]$. In fact, a descent inside $u$ appears exactly when a number in $[m]$ follows behind one in $[m+1,m+n]$. Hence, the number
of $u$'s determining compositions of length $\ell$ is
\[\sum_{i=0}^{n+1-\ell}{m \choose \ell-1}{n-1-i \choose \ell-2}={m \choose \ell-1}{n \choose \ell-1},\]
where the index $i$ counts how many numbers in $[m+1,m+n]$ are at the tail of $u$. Hence,
\begin{align*}
F_{0^m}&\cdot F_{0^n}=\sum_{\ell=1}^{m+1}
{m \choose \ell-1}{n \choose \ell-1}
\sum_{j=0}^{\ell-1}(-1)^j{\ell-1 \choose j}F_{0^{m+n-j}}\\
&=\sum_{j=0}^m(-1)^j\lb\sum_{\ell=j+1}^{m+1}
{m \choose \ell-1}{n \choose \ell-1}
{\ell-1 \choose j}\rb F_{0^{m+n-j}}\\
&=\sum_{j=0}^m(-1)^j{m \choose j}\lb\sum_{\ell=j+1}^{m+1}
{n \choose \ell-1}{m-j \choose m-\ell+1}\rb F_{0^{m+n-j}}
=\sum_{j=0}^m(-1)^j{m \choose j}{m+n-j \choose m}F_{0^{m+n-j}}.\qedhere
\end{align*}
\end{proof}

\noindent For instance,
\begin{align*}
&F_{0^n}=\sum_{i=1}^n{n-1 \choose i-1}M_{0^i},\\
&F_0\cdot F_{0^n}
=(n+1)F_{0^{n+1}}-nF_{0^n}=\sum_{i=1}^{n+1}{n\choose i-1}iM_{0^i},\\
&F_{0^2}\cdot F_{0^n}={n+2\choose 2}F_{0^{n+2}}-2{n+1\choose 2}F_{0^{n+1}}+{n\choose 2}F_{0^n}\\
\end{align*}
for any $n\geq1$.

Consequently, one can expect that the multiplication of two WPFQFs is not necessarily a positive linear combination of WPFQFs. To study the multiplication rule of WPFQFs, we should first show that they are closed under the usual product. It forces us to provide the following technical lemma.
\begin{lem}
For any composition $\al$ of $n$, the function $\La(P_\al,\emptyset)$ has the following expression,
\beq\label{pp}
\La(P_\al,\emptyset)=\sum_{j=p+1}^n
m_{\al,j}2^jM_{0^j},
\eeq
where $m_{\al,j}=|\{\be\unlhd\al\,|\,\ell(\be)=j\}|=\sum\limits_{i=0}^p{n-2p-1 \choose i+j-2p-1}{p \choose i}2^i$ with $p=|P(\al)|$.
Moreover, $\La(P_\al,\emptyset)$ can be $\bZ$-linearly expanded by $K_{0^j}$'s as follows.
\beq\label{pk}
\La(P_\al,\emptyset)=\sum_{j=0}^{n-1}a_j K_{0^{n-j}}=\sum_{k=0}^p(-1)^k{p\choose k}K_{0^{n-2k}},
\eeq
where the coefficients
\[a_k=\sum_{j=0}^k(-1)^{k-j}{n-1-j\choose n-1-k}m_{\al,n-j}=
\begin{cases}
(-1)^{k/2}{p\choose k/2},&\mb{if }k\mb{ is even},\\
0,&\mb{otherwise},
\end{cases}\]
for $k=0,1\dots,n-1$.
\end{lem}
\begin{proof}
First note that for any $r\in\bP$,
\[K_{0^r}=\sum_{j=1}^r2^j{r-1\choose j-1}M_{0^j}=
\sum_{j=1}^r(-1)^{r-j}2^j{r-1\choose j-1}F_{0^j}.\]
Also, one can easily check the following inversion formulas,
\beq\label{ik}
M_{0^r}=2^{-r}\sum_{i=1}^r(-1)^{r-i}{r-1 \choose i-1}K_{0^i}\mb{ and }F_{0^r}=2^{-r}\sum_{i=1}^r{r-1 \choose i-1}K_{0^i}.
\eeq

For $\al\vDash n$, taking $S=[n]$, we have $\La(P_\al,S)=K_\al$ by definition, and $\La(P_\al,\emptyset)$ is the opposite extreme case.
According to \eqref{ka} and \eqref{PM} we have
\[\La(P_\al,\emptyset)=\sum_{\be\unlhd\al}2^{\ell(\be)}M_{0^{\ell(\be)}}=\sum_{j=p+1}^n
m_{\al,j}2^jM_{0^j},\]
where $m_{\al,j}=|\{\be\unlhd\al\,|\,\ell(\be)=j\}|$ for $j=1,\dots,n$. In particular, $m_{\al,n}=1$ and $m_{\al,j}=0$ if $j\leq p$.

For any $\be\unlhd\al$, it follows exactly one rule that the descent set $D(\be)$ must contain either $i-1$ or $i$ when $i\in P(\al)$. Hence, if $\ell(\be)=j$, we have
\[m_{\al,j}=\sum_{i_1,\dots,i_p\in\{1,2\}}
{n-2p-1 \choose j-i_1-\cdots-i_p-1}2^{2p-i_1-\cdots-i_p}
=\sum_{i=0}^p{n-2p-1 \choose i+j-2p-1}{p \choose i}2^i.\]

Now by \eqref{ik}, we have
\begin{align*}
&\La(P_\al,\emptyset)=\sum_{j=p+1}^n
m_{\al,j}2^jM_{0^j}=\sum_{j=p+1}^nm_{\al,j}\sum_{s=1}^j(-1)^{j-s}{j-1 \choose s-1}K_{0^s}\\
&\quad=\sum_{s=1}^n\lb\sum_{j=s}^n(-1)^{j-s}{j-1 \choose s-1}m_{\al,j}\rb K_{0^s}
=\sum_{s=0}^{n-1}\lb\sum_{j=0}^s(-1)^{s-j}{n-1-j \choose n-1-s}m_{\al,n-j}\rb K_{0^{n-s}}.
\end{align*}

Consequently, we set
\[a_k:=\sum_{j=0}^k(-1)^{k-j}{n-1-j\choose n-1-k}m_{\al,n-j}\]
for $k=0,1\dots,n-1$. Then there exists a lower unitriangular matrix $C=\lb n-j\choose n-i\rb_{i,j=1,\dots,n}$, containing Pascal's triangle, such that
\beq\label{mm}C\cdot\begin{pmatrix}
a_0\\
\vdots\\
a_{n-1}
\end{pmatrix}=
\begin{pmatrix}
m_{\al,n}\\
\vdots\\
m_{\al,1}
\end{pmatrix},\quad
\begin{pmatrix}
a_0\\
\vdots\\
a_{n-1}
\end{pmatrix}=
C^{-1}\cdot\begin{pmatrix}
m_{\al,n}\\
\vdots\\
m_{\al,1}
\end{pmatrix},\eeq
since $C^{-1}=\lb(-1)^{i-j}{n-j\choose n-i}\rb_{i,j=1,\dots,n}$. Actually, for any $i,j=1,\dots,n$,
\[\sum_{k=1}^n{n-k\choose n-i}\cdot(-1)^{k-j}{n-j\choose n-k}=
{n-j\choose n-i}\sum_{j\leq k\leq i}(-1)^{k-j}{i-j\choose k-j}={n-j\choose n-i}(1-1)^{i-j}=\de_{ij}.\]

Next we prove that
\[a_k=\begin{cases}
(-1)^{k/2}{p\choose k/2},&\mb{if }k\mb{ is even},\\
0,&\mb{otherwise},
\end{cases}\]
which is equivalent to the following combinatorial identities by \eqref{mm},
\[\sum_{i=0}^{\lfloor k/2\rfloor}(-1)^i{n-1-2i\choose n-1-k}{p\choose i}=m_{\al,n-k}=\sum_{i=0}^p{n-2p-1\choose n-2p-1+i-k}{p\choose i}2^i\]
for $k=0,1\dots,n-1$, where $\lfloor k/2\rfloor$ is the largest integer $\leq k/2$.
Here we only show the cases for even $k$, and the cases for odd $k$ are similar. Let $k=2r$, it becomes
\beq\label{com}\sum_{i=0}^p(-1)^i{n-1-2i\choose n-1-2r}{p\choose i}=\sum_{i=0}^p{n-2p-1\choose 2r-i}{p\choose i}2^i.\eeq

The explanation for such identity \eqref{com} is given by the following combinatorial proof. For fixed $0\leq2p,2r\leq n-1$, one first marks $2p$ red cards and then $n-1-2p$ blue cards with numbers from $1$ to $n-1$. Now pick up $2r$ cards among them with restriction that for each $l\in[p]$, at most one of two red cards marked $2l-1,2l$ can be chosen. It is clear that the RHS of \eqref{com} gives the number of choices of cards under such constraint.

Alternately, these choices can be made by rejecting the rest $n-1-2r$ cards with at least one of two red cards marked $2l-1,2l$ to be dropped for each $l\in[p]$. Given any $i=0,1,\dots,p$, one can first select $i$ pairs of red cards marked $2l-1,2l$ for $l\in[p]$, then exclude $n-1-2r$ cards from the rest $n-1-2i$ ones. Denote $c_i$ the number of ways for such operation. By the inclusion-exclusion principle, $\sum_{i=0}^p(-1)^ic_i$, exactly the LHS of \eqref{com}, counts the same choices as previous. Then we verify identity \eqref{com}, thus formula \eqref{pk}.
\end{proof}

\begin{lem}
Given $m,n\in\bP$, we have
\beq\label{mukz}
K_{0^m}\cdot K_{0^n}=\sum_{k=0}^m(-1)^k {m+n-k \choose m}{m \choose k}\dfrac{m+n-2k}{m+n-k}
K_{0^{m+n-2k}}.
\eeq
\end{lem}
\begin{proof}
We can follow the idea in the proof of Lemma \ref{mufz}. Let $m\leq n$, and we know that
\[K_{0^m}\cdot K_{0^n}=\La([m],\emptyset)\cdot\La([n],\emptyset)=\sum_{u\in[m]\shuffle[m+1,m+n]}\La(u,\emptyset).\]
Now any $u\in[m]\shuffle[m+1,m+n]$ is isomorphic to one $P_\al$, and
the cardinality $|P(\al)|$ determines $\La(u,\emptyset)$ by \eqref{pk}. We need to count how many $u$'s correspond to compositions with $p$ peaks for any given $p=0,1,\dots,m$. In fact, a peak of $u$ appears just when a number in $[2,m]$ follows behind one in $[m+1,m+n]$, or the number $1$ follows behind at least two numbers in $[m+1,m+n]$. Therefore, the number
of $u$'s determining compositions with $p$ peaks is
\[\sum_{i=0}^{n-p}{m-1 \choose p}{n-1-i \choose p-1}
+\sum_{j=0}^{n-p-1}{m \choose p}{n-2-j \choose p-1}
={m-1 \choose p}{n \choose p}
+{m \choose p}{n-1 \choose p},\]
where the index $i$ counts how many numbers in $[m+1,m+n]$ are at the tail of $u$, when the number $1$ is at the head of $u$. Otherwise, we use the index $j$. Again by \eqref{pk}, we have
\begin{align*}
K_{0^m}\cdot K_{0^n}
&=\sum_{p=0}^m\lb{m-1 \choose p}{n \choose p}
+{m \choose p}{n-1 \choose p}\rb
\sum_{k=0}^p(-1)^k{p\choose k} K_{0^{m+n-2k}}\\
&=\sum_{k=0}^m(-1)^k \lb\sum_{p=k}^m{m-1 \choose p}{n \choose p}{p\choose k}+{m \choose p}{n-1 \choose p}{p\choose k}\rb K_{0^{m+n-2k}}\\
&=\sum_{k=0}^m(-1)^k \lb\sum_{p=k}^m{m-1 \choose k}{m-1-k\choose m-1-p}{n \choose p}+{m \choose k}{m-k \choose m-p}{n-1 \choose p}\rb K_{0^{m+n-2k}}\\
&=\sum_{k=0}^m(-1)^k \lb{m-1 \choose k}{m+n-1-k \choose m-1}+{m \choose k}{m+n-1-k \choose m}\rb
K_{0^{m+n-2k}}\\
&=\sum_{k=0}^m(-1)^k {m+n-k \choose m}{m \choose k}\dfrac{m+n-2k}{m+n-k}
K_{0^{m+n-2k}}.\qedhere
\end{align*}
\end{proof}

By formulas \eqref{mul}, \eqref{pk}, the multiplication of $K_{0^r}$'s can be $\bZ$-linearly expanded in terms of themselves. Similarly, WPQSym has a subalgebra spanned by $K_{0^r},\,\forall r\in\bP$, with $K_{0^0}:=K_\emptyset=1$, and we denote it by WPQSym$^0$. Moreover,  WPQSym$^0_\bQ=$WCQSym$^0_\bQ$ by \eqref{ik}, with two bases $\{F_{0^r}\}_{r\in\bN}$ and $\{K_{0^r}\}_{r\in\bN}$.
\begin{exam} We list some cases of low degree below.
\begin{align*}
&K_0=\La(P_1,\emptyset)=2M_0, K_{0^2}=\La(P_2,\emptyset)=4M_{0^2}+2M_0,\\
&K_{0^3}=\La(P_3,\emptyset)=8M_{0^3}+8M_{0^2}+2M_0,\\ &K_{0^4}=\La(P_4,\emptyset)=16M_{0^4}+24M_{0^3}+12M_{0^2}+2M_0,\\
&K_0\cdot K_0=\La(P_2,\emptyset)+\La(P_{11},\emptyset)=2K_{0^2},\\
&K_{0^2}\cdot K_0=
\La(P_3,\emptyset)+\La(P_{21},\emptyset)+\La(P_{12},\emptyset)=3K_{0^3}-K_0,\\
&K_{0^3}\cdot K_0=
\La(P_4,\emptyset)+\La(P_{31},\emptyset)+\La(P_{22},\emptyset)+\La(P_{13},\emptyset)=4K_{0^4}-2K_{0^2},\\
&K_0\cdot K_{0^n}
=(n+1)K_{0^{n+1}}-(n-1)K_{0^{n-1}},\\
&K_{0^2}\cdot K_{0^n}=
{n+2 \choose 2}K_{0^{n+2}}-n^2K_{0^n}+(n-2)(n-1)K_{0^{n-2}},
\end{align*}
for any $n\geq1$.
\end{exam}

Let $\mP\mC(n):=\{\al\in\mW\mC(n)\,|\,P(\al)=D(\al)\backslash\{\|\al\|\}\}$ and $\mP\mC=\bigcup\limits_{n\geq0}^.\mP\mC(n)$. For the ring structure of WPQSym, we give the following main result by an effort.
\begin{theorem}\label{str}
WPQSym is a graded subring of WCQSym. Moreover for $n\geq0$, WPQSym$_n$ is freely generated as a $\bZ$-module by those $K_\al$'s with $\al\in\mP\mC(n)$.
\end{theorem}

\begin{proof}
We fix the weight $n$ and totally order $\mW\mC(n)$ as follows, first by decreasing order of the 0-length $\ell_0$, i.e. the total weight $\|\cdot\|$, then by left-to-right counting of zeros, and finally by reverse lexicographic order on the descent sets. That is, if we abuse to denote such order as $<$, then
\[\cdots<(k,0,n-k)<(k,0,n-k-1,1)<\cdots<(0,1^n)<n<(n-1,1)<\cdots<(1^n)\]
Let $\mK_n:=\{K_\al\}_{\al\in\mP\mC(n)}$.
By Lemma \ref{eq}, for any $\be\in\mW\mC(n)$, there exists a unique $\al\in\mP\mC(n)$ such that $K_\al=K_\be$. Hence, $\mK_n$ also spans WPQSym$_n$. Note that if $\ga\unlhd\al$, we have $\ell_0(\ga)\leq\ell_0(\al)$. For any $\al\in\mP\mC(n)$ without three consecutive parts $0,1,r-1$ with $r\geq2$, the expansion of $K_\al$ in terms of the WCMQFs as \eqref{PM} has the smallest term $2^{\ell(\al)}M_\al$ with respect to the above total order on $\mW\mC(n)$. Thus such $K_\al$'s form an linear independent subset of $\mK_n$.

For the linear independence of $\mK_n$, it is enough to check the special case $\{K_{(0^i,\al)}:\al\vDash n,\,(0^i,\al)\in\mP\mC(n),i\in\bN\}$ for given $n\geq2$.
Indeed,
\[K_{(0^i,n)}=\sum_{j=0}^i\sum_{\al\vDash n}{i\choose j}2^{j+\ell(\al)}M_{(0^j,\al)}=K_{(0^i,1,n-1)}+
\sum_{j=0}^{i-1}\sum_{p=2}^n\sum_{\be\vDash n-p}{i\choose j}2^{j+1+\ell(\be)}M_{(0^j,p,\,\be)}.\]
Note that all $K_{(0^i,\al)},\,(0^i,\al)\in\mP\mC(n)$, except $K_{(0^i,n)}$ and $K_{(0^i,1,n-1)}$, lack the term $M_{(0^i,n)}$. Suppose that there exists a nontrivial vanishing $\bZ$-linear combination of these $K_{(0^i,\al)}$'s, then we can find the largest $m\in\bP$ such that
\[K_{(0^m,n)}-K_{(0^m,1,n-1)}\in\bigoplus_{0\leq r\leq m-1}\bigoplus_{\be\vDash n\atop(0^r,\,\be)\in\mP\mC(n)}\bQ K_{(0^r,\,\be)}.\]
Among these WPFQFs, only
\[K_{(0^m,n)},\,K_{(0^m,1,n-1)},\,K_{(0^{m-1},n)},\,K_{(0^{m-1},s,n-s)},\,K_{(0^{m-1},1,t,n-1-t)},\,s\in[1,n-1],t\in[2,n-2],\]
can be expanded with some of $M_{(0^{m-1},n)},\,M_{(0^{m-1},s,n-s)}$ as terms.

On the other hand, for any $a_0,a_1\dots,a_{n-1},b_2,\dots,b_{n-2}\in\bQ$ with $b_1=0$,
\begin{align*}
a_0K&_{(0^{m-1},n)}+\sum_{s=1}^{n-1}a_s K_{(0^{m-1},s,n-s)}+\sum_{t=2}^{n-2}b_t K_{(0^{m-1},1,t,n-1-t)}\\
&=(a_0+a_1)\lb2^mM_{(0^{m-1},n)}+2^{m+1}\sum_{s=2}^{n-1}M_{(0^{m-1},s,n-s)}\rb
+(a_0+a_1+a_2)2^{m+1}M_{(0^{m-1},1,n-1)}\\
&+2^{m+1}\lb\sum_{s=2}^{n-2}(a_s+a_{s+1}+b_{s-1}+b_s)M_{(0^{m-1},s,n-s)}+(a_{n-1}+b_{n-2})M_{(0^{m-1},n-1,1)}\rb\\
&+\mb{ terms of }M_{(0^j,\,\be)}\mb{ with }0\leq j\leq m-1,\ell(\be)>2.
\end{align*}
It necessarily takes $a_0+a_1=-m$, such that
\[K_{(0^m,n)}-K_{(0^m,1,n-1)}+a_0K_{(0^{m-1},n)}+\sum_{s=1}^{n-1}a_s K_{(0^{m-1},s,n-s)}+\sum_{t=2}^{n-2}b_t K_{(0^{m-1},1,t,n-1-t)}\]
has none of terms $M_{(0^{m-1},n)}$ and $M_{(0^{m-1},s,n-s)},\,s\in[1,n-1]$. However, this is impossible, as we also need $a_s+b_{s-1}=(-1)^sm$ for $s\geq2$, and then the term $M_{(0^{m-1},n-1,1)}$ remains.
Hence, we get the contradiction. It implies that $\mK_n$ is a linear independent set spanning WPQSym, and thus a $\bZ$-basis of WPQSym.


Next we give the multiplication rule of $K_\al$'s. As a special case of \eqref{mul}, for two weak compositions $\al,\be$,
\beq\label{muk}
K_\al\cdot K_\be=\La(P_\al)\cdot\La(P_\be)=\sum_{u\in(P_\al\sim\be)\shuffle(P_\be\smile\al)}\La(u),
\eeq
using the notation in \eqref{muf}. We need to prove that for any chain $u\in(P_\al\sim\be)\shuffle(P_\be\smile\al)$, $\La(u)$ can be written as a $\bZ$-linear combination of the WPFQFs.

In fact, the function $\La(u)$ can be expressed as follows. By contrast with those $(P_\ga,S)$'s, $u$ may have decreasing consecutive numbers under an overline. Consequently, the descents of any interval in $u$ under an overline determine a composition $\ga$. Note that the numbers with overlines in $u$ are smaller than other without lines. Thus for any such interval of $u$ with composition $\ga$, we can substitute the formula \eqref{pk} of $\La(P_\ga,\emptyset)$ into the relative part in $\La(u)$, to derive the expansion of $\La(u)$ in terms of the WPFQFs. For example, if $\al=(1,0^21,0),\,\be=(0,1,0^2)$, write $P_\al=5\ol{12}4\ol{3},\,P_\be=\ol{1}4\ol{23}$, we see that $P_\al\sim\be=8\ol{12}7\ol{3}$ and $P_\be\smile\al=\ol{4}9\ol{56}$. In particular, $u=8\ol{142}97\ol{563}\in(P_\al\sim\be)\shuffle(P_\be\smile\al)$. In addition,  $\ol{142},\,\ol{563}$ both correspond to $\ga=21$ with $P_\ga=132$. We have
\[\La(u)=K_{10^31^20^3}-K_{10^31^20}-K_{101^20^3}+K_{101^20},\]
as $\La(132,\emptyset)=K_{0^3}-K_0$. As a result, WPQSym is a graded subring of WCQSym.
\end{proof}

\begin{exam}
For $K_2=2M_2+4M_{11},\,K_{10}=4M_{10}+2M_2,\,K_{01}=4M_{01}+2M_1$, as
\begin{align*}
&K_{201}=8M_{201}+16M_{1101}+4M_{21}+4M_{12}+16M_{111},\\
&K_{102}=8M_{102}+16M_{1011}+2M_3+4M_{21}+8M_{12}+16M_{111},\\
&K_{1011}=8M_{102}+16M_{1011}+4M_{21}+4M_{12}+16M_{111},\\
&K_{03}=4M_{03}+8M_{021}+8M_{012}+16M_{0111}+2M_3+4M_{21}+4M_{12}+8M_{111},\\
&K_{021}=8M_{021}+8M_{012}+16M_{0111}+4M_{21}+4M_{12}+8M_{111},\\
&K_{012}=4M_{03}+8M_{021}+8M_{012}+16M_{0111}+4M_{12}+8M_{111},
\end{align*}
we have
\begin{align*}
K_2\cdot K_{01}&=8(M_2\cdot M_{01})+4(M_2\cdot M_1)+16(M_{11}\cdot M_{01})+8(M_{11}\cdot M_1)\\
&=8(M_{201}+M_{03}+M_{021}+M_{012}+M_{21})+4(M_3+M_{21}+M_{12})\\
&+16(M_{1101}+M_{102}+2M_{1011}+M_{021}+M_{012}+3M_{0111}+M_{12}+3M_{111})\\
&+8(M_{21}+M_{12}+3M_{111})\\
&=\La(23\ol{1}4)+\La(2\ol{1}34)+\La(2\ol{1}43)+\La(\ol{1}234)+\La(\ol{1}342)+\La(\ol{1}324)\\
&=K_{201}+K_{102}+K_{1011}+K_{03}+K_{021}+K_{012}.\\
K_{00}\cdot K_{01}&=16(M_{00}\cdot M_{01})+8(M_{00}\cdot M_1)+8(M_0\cdot M_{01})+4(M_0\cdot M_1)\\
&=16(M_{0100}+2M_{0010}+3M_{0001}+2M_{010}+4M_{001}+M_{01})\\
&+8(M_{100}+M_{010}+M_{001}+M_{10})+M_{01}+8(M_{010}+2M_{001}+2M_{01})\\
&+4(M_{10}+M_{01}+M_1)\\
&=\La(\ol{3}4\ol{12})+\La(\ol{13}4\ol{2})+\La(\ol{31}4\ol{2})+\La(\ol{123}4)
+\La(\ol{312}4)+\La(\ol{132}4)\\
&=K_{0100}+2K_{0010}+3K_{0001}-K_{01}.
\end{align*}
\end{exam}

\section{The Hopf algebra of weak peak quasisymmetric functions}
The existence of Hopf algebra structure on WPQSym will be studied in this section.

\subsection{WCQSym and WPQSym as Hopf algebras}
According to \cite[Prop. 2.5]{GTY}, the authors have defined the Hopf algebra structure on WCQSym as follows.
\beq\label{CM}
\De(M_\al)=\sum_{\be\ga=\al}M_\be\ot M_\ga=\sum_{i=0}^{\ell}M_{(n_1,\dots,n_i)}\ot M_{(n_{i+1},\dots,n_\ell)},\,\ve(M_{\al})=\de_{\al,\emptyset}
\eeq
and
\[S(M_\al)=(-1)^\ell\sum_{J\vDash\ell}M_{J\circ\al^r}\]
for any $\al=(n_1,\dots,n_\ell)\in\mW\mC$ with all $n_i\in\bN$, where
\[J\circ\al^r=(n_1+\cdots+n_{j_1},n_{j_1+1}+\cdots+n_{j_1+j_2},\dots,n_{j_1+\cdots+j_{l-1}+1}+\cdots+n_\ell)\]
for any $J=(j_1,\dots,j_l)\vDash\ell$. In particular, $S(M_{0^n})=(-1)^n\sum\limits_{i=1}^n{n-1\choose i-1}M_{0^i}=(-1)^nF_{0^n}$, thus $S(F_{0^n})=\sum\limits_{i=1}^n{n-1\choose i-1}(-1)^iF_{0^i}=(-1)^nM_{0^n}$.
On the other hand, the comultiplication and antipode formulas of $F_\al$'s in QSym can be found in \cite{Eh,MR}. Here we generalize it to the weak case.
\begin{prop}\label{CF}
For any $\al\in\mW\mC$,
\[\De(F_\al)=\sum_{\be|\ga=\al}F_\be\ot F_\ga,\, \ve(F_\al)=\de_{\al,\emptyset},\,S(F_\al)=\sum_{\ol{\be}=\ol{\al^t}}(-1)^{\|\be\|}d_{\al^t,\,\be}F_\be,\]
where we write $\be|\ga=\al$ if $\be\ga=\al$, or if the last part of $\be$ and the first one of $\ga$ are positive and $\be\vee\ga=\al$. The coefficient $d_{\al\be}={i_1-1\choose j_1-1}\cdots{i_k-1 \choose j_k-1}{i_{k+1}-1\choose j_{k+1}-1}$ for any $\al=(0^{i_1},s_1,\dots,0^{i_k},s_k,0^{i_{k+1}})$ and $\be=(0^{j_1},s_1,\dots,0^{j_k},s_k,0^{j_{k+1}})$ with all $i_p,j_p\in\bN$ and $s_q\in\bP$.
\end{prop}
\begin{proof}
By formulas \eqref{FM} and \eqref{CM}, we have
\[\De(F_\al)=\sum_{\eta\leq\al}c_{\al\eta}\De(M_\eta)=\sum_{\eta\leq\al}c_{\al\eta}\sum_{\xi\zeta=\eta}M_\xi\ot M_\zeta=\sum_{\xi\zeta\leq\al}c_{\al,\,\xi\zeta}M_\xi\ot M_\zeta.\]
Fix $\xi,\zeta$ such that $\xi\zeta\leq\al$. If the last part of $\xi$ is positive, then by definition there exist unique $\be,\ga\in\mW\mC$ such that $\be|\ga=\al,\,\xi\leq\be,\,\zeta\leq\ga$ and $c_{\al,\,\xi\zeta}=c_{\be\xi}c_{\ga\zeta}$.

Otherwise, if the last part of $\xi$ is zero, there exist pairs  $\be,\ga\in\mW\mC$ such that $\be|\ga=\be\ga=\al,\,\xi\leq\be,\,\zeta\leq\ga$ instead.
Suppose that the last interval of zeros in $\xi$ is $0^k$. It joins into a interval $0^j$ in $\xi\zeta$, as a contraction of $0^i$ in $\al$. In particular, we have $0<k\leq j\leq i$. In order to obtain the desired identity \[c_{\al,\,\xi\zeta}=\sum_{\be\ga=\al\atop\xi\leq\be,\,\zeta\leq\ga}c_{\be\xi}c_{\ga\zeta},\] we only need to prove that
\beq\label{c1}
{i\choose j}=\sum_{s=k}^{i-j+k}{s-1\choose k-1}{i-s\choose j-k}
=\sum_{p=0}^{i-j}{k-1+p\choose k-1}{i-k-p\choose j-k},
\eeq
when such $0^i$ in $\al$ is not at the tail. Otherwise, one can replace $i,j$ by $i-1,j-1$ respectively to modify it. Using
\[(1-x)^{-r}=\lb1+x+x^2+\cdots\rb^r=\sum_{n\geq0}{r+n-1\choose r-1}x^n,\,r\in\bP\]
and comparing the coefficients of $x^n$ in both sides of
\[(1-x)^{-r_1-r_2}=(1-x)^{-r_1}(1-x)^{-r_2},\]
we have
\beq\label{c3}
{r_1+r_2+n-1\choose r_1+r_2-1}=\sum_{p=0}^n{r_1+p-1\choose r_1-1}{r_2+n-p-1\choose r_2-1},\,r_1,r_2\in\bP,n\in\bN,
\eeq
useful for the rest paper.
Now pick $r_1=k,\,r_2=j-k+1,\,n=i-j$ to get \eqref{c1}. In conclusion,
\[\De(F_\al)=\sum_{\xi\zeta\leq\al}c_{\al,\,\xi\zeta}M_\xi\ot M_\zeta
=\sum_{\be|\ga=\al}\sum_{\xi\leq\be}c_{\be\xi}M_\xi\ot \sum_{\zeta\leq\ga}c_{\ga\zeta}M_\zeta=\sum_{\be|\ga=\al}F_\be\ot F_\ga.\]

For the formula of $S$ on the WCFQFs in WCQSym, we adopt the method of sign-reversing involutions introduced in \cite{BS} to obtain the desired cancellation-free one.

For any connected, graded Hopf algebra $H=\oplus_{n\geq0}H_n$, let $\pi=\mb{id}-u\circ\ve$,\,$u:k\rw H$ is the unit map. Then $\pi$ is locally nilpotent with respect to convolution, thus id$=\pi+u\circ\ve$ is invertible, and it gives the following convolution formula for antipode $S$ due to Takeuchi,
\beq\label{anti}
S=\sum_{k\geq0}(-1)^k\pi^{\star k}=u\circ\ve+\sum_{k\geq1}(-1)^k m^{k-1}\circ\pi^k\circ\De^{k-1},
\eeq
where $\star$ is the convolution, and $\pi^{\star0}=u\circ\ve$; see also \cite[Section 5]{AS}.

By \eqref{CM}, WCQSym can be a connected, graded coalgebra with $\bN$-grading defined by deg\,$M_\al:=\ell(\al)$, though its multiplication is not compatible with such grading.
In particular, the map $\pi$ is still locally nilpotent with respect to convolution, thus formula \eqref{anti} also works for WCQSym.

Referring to the argument in \cite[Section 5]{BS} for the fundamental basis $F_\al$ of QSym, we begin to handle the case of WCFQFs in WCQSym.
Given nonempty $\al\in\mW\mC$, we use the following rim-hook table to present the labeled poset $P_\al$ with the default $S=\cup_pP_p$. Assume that $D(\al)=\{a_1,\dots,a_k\}$. First write a diagram with $a_i$ boxes in the $i$th row from the bottom, and the last box of the $i$th row is in the same column as the first box of the $(i+1)$th row. Then we fill entries of this diagram from the bottom row to the top one, and from left to right along each row by the one-line word representing $P_\al$. The resulting table is called the \textit{rim-hook tableau} of $\al$, denoted $T$, with the previous filling sequence as the \textit{row word} $w_T$.
In particular, the descents of $\al$ in $D(\al)$ appear exactly at the corners of its diagram.

For example, if $\al=(2,0^3,1,0)$, $P_\al$ has the one-line-word presentation $w_T=67\ol{123}5\ol{4}$, and the associated rim-hook tableau
\[T=\raisebox{-1.2em}{\xy 0;/r.16pc/:
(0,0)*{};(10,0)*{}**\dir{-};
(0,5)*{};(25,5)*{}**\dir{-};
(5,10)*{};(25,10)*{}**\dir{-};
(20,15)*{};(25,15)*{}**\dir{-};
(0,0)*{};(0,5)*{}**\dir{-};
(5,0)*{};(5,10)*{}**\dir{-};
(10,0)*{};(10,10)*{}**\dir{-};
(15,5)*{};(15,10)*{}**\dir{-};
(20,5)*{};(20,15)*{}**\dir{-};
(25,5)*{};(25,15)*{}**\dir{-};
(2.5,2.5)*{\stt{6}};(7.5,2.5)*{\stt{7}};(7.5,7.5)*{\stt{\ol{1}}};(12.5,7.5)*{\stt{\ol{2}}};
(17.5,7.5)*{\stt{\ol{3}}};(22.5,7.5)*{\stt{5}};
(22.5,12.5)*{\stt{\ol{4}}};
\endxy}\,.\]
If we can write $\al=\al_1|\cdots|\al_k,\,\al_1,\dots,\al_k\neq\emptyset$, it means that $T$ can divide into a disjoint union of subtableaux $T_i$ of shape $\al_i$ with $w_i:=w_{T_i},\,i=1,\dots,k$, such that $w_T$ is the concatenation $w_1\cdots w_k$. For a word $v$, the \textit{restriction} of $v$ to a set $B$ of letters is a subword $v|_B$ of $v$ with letters in $B$.

Now we define a sign-reversing involution $\iota$ on the set
\[\Om_\al:=\{(x,v)\,|\,v\in x=w_1\shuffle\cdots\shuffle w_k\mb{ with }w_1\cdots w_k=w_T,\,T_1,\dots,T_k\neq\emptyset\}\]
as follows. Let sgn\,$(x,v)=(-1)^k$.
Consider a word $v$ in the shuffle product $x=w_1\shuffle\cdots\shuffle w_k$, and find the smallest index $j$ if possible, such that

(1) all the lengths of $w_1,\dots,w_j$ are 1;

(2) if $B=w_jw_{j+1}$, then $v|_B$ is the row word of a rim-hook subtableau of $T$.

If this is the case, we do the merge operation, and let $\iota(x,v)=(x',v)$, where
\[x'=w_1\shuffle\cdots w_{j-1}\shuffle w_jw_{j+1}\shuffle\cdots w_k.\]
In particular, sgn$\,(x',v)=-$sgn$\,(x,v)$. From the example above, if $x=6\shuffle7\shuffle\ol{123}\shuffle5\ol{4}$ and $v=576\ol{1243}\in x$, then we have $j=2$, $B=\{\ol{1},\ol{2},\ol{3},7\}$ and $v|_B=7\ol{123}$ is the row word of the middle subtableau of $T$.
Thus $\iota(x,v)=(x',v)$ with $x'=6\shuffle7\ol{123}\shuffle5\ol{4}$.

Otherwise, if the number $k<\|\al\|$ and any index satisfying conditions (1), (2) above can not be found, then there exists the smallest index $j$ such that

(1) the length of $w_j\geq2$;

(2) if $j>1$, then the sole element of $w_{j-1}$ is to the right of the leftmost one of $w_j$ in $v$.

In this case, we do splitting and let $\iota(x,v)=(x',v)$, where
\[x'=w_1\shuffle\cdots w_j\shuffle d\shuffle w'_j\shuffle w_{j+1}\shuffle\cdots w_k\]
and $w_j=dw'_j$. Clearly sgn$\,(x',v)=-$sgn$\,(x,v)$. As in the previous example, if $x=6\shuffle7\ol{123}\shuffle5\ol{4}$ and $v=576\ol{1243}\in x$, then we have $j=2$, and the sole element $6$ of $w_1$ is to the right of the leftmost $7$ of $w_2$ in $v$. Thus $\iota(x,v)=(x',v)$ with $x'=6\shuffle7\shuffle\ol{123}\shuffle5\ol{4}$.

Finally, if $(x,v)$ does not belong to both two cases above, then we must have $k=\|\al\|$, i.e. all the $w_i$ are of length 1, and every pairs of adjacent letters in $w_T$ are in the reverse order in $v$, thus $v$ is the reverse of $w_T$. For this last term, we let it be the unique fixed point of $\iota$, namely $\iota(x,v)=(x,v)$. It is clear that $\iota$ is a sign-reversing involution on $\Om_\al$.

Let $\pi=id-u\circ\ve$ satisfying $\pi(F_\al)=(1-\de_{\al,\emptyset})F_\al$. By the comultiplication rule for $F_\al$ and \eqref{anti}, we have
\[S(F_\al)=\sum_{k\geq1}(-1)^k\sum_{\al_1,\dots,\al_k\neq\emptyset\atop
\al_1|\cdots|\al_k=\al}F_{\al_1}\cdots F_{\al_k}.\]
Further by \eqref{muf} and the construction of $\iota$, we know that
\[S(F_\al)=\sum_{k\geq1}(-1)^k\sum_{T_1,\dots,T_k\neq\emptyset\atop
w_1\cdots w_k=w_T}\Ga(w_1)\cdots \Ga(w_k)=\sum_{(x,v)\in\Om_\al}\mb{sgn}\,(x,v)\Ga(v)=(-1)^{\|\al\|}\Ga(P_\al^*),\]
where $P_\al^*$ is the dual of $P_\al$, namely with the opposite partial order but the same labelling. Its one-line-word presentation can be obtained by reversing the original one. In order to expand $\Ga(P_\al^*)$ in terms of $F_\be$'s, we need to eliminate all opposite order of consecutive labels under overlines in $P_\al^*$. Indeed for each interval $\ol{u_1\dots u_r}$ in $P_\al^*$, $u_1>\cdots>u_r$, we can
substitute the formula from \eqref{gf},
\[\Ga(P_{1^r},\emptyset)=\sum_{j=0}^{r-1}(-1)^j{r-1\choose j}F_{0^{r-j}},\]
into its relative part in $\Ga(P_\al^*)$, and get our desired formula. Again for the example $\al=(2,0^3,1,0)$, $P_\al=67\ol{123}5\ol{4}$, then $P_\al^*=\ol{4}5\ol{321}76$, and we have
\[S(F_\al)=S(F_{20^310})=(-1)^7\Ga(\ol{4}5\ol{321}76)=-F_{010^31^2}+2F_{010^21^2}-F_{0101^2},\]
since $\Ga(P_{1^3},\emptyset)=F_{0^3}-2F_{0^2}+F_0$.
\end{proof}

As a special case, we obtain that
\[\De(F_{0^n})=\sum_{i=0}^n F_{0^i}\ot F_{0^{n-i}}\]
with $F_{0^0}=F_\emptyset=1$ as usual, and $S(F_{0^n})=(-1)^nM_{0^n}$. Besides, we note that $S^2=$id on WCQSym. Now we are in the position to discuss the Hopf algebra structure on WPQSym.

\begin{theorem}\label{hk}
WPQSym is a Hopf subalgebra of WCQSym, as
\[\De(K_\al)=\sum_{\be|\ga=\al}K_\be\ot K_\ga,\,\ve(K_{\al})=\de_{\al,\emptyset}\]
and
\[S(K_\al)=(-1)^{\|\al\|}K_{\al^t}\]
for any $\al\in\mW\mC$.
\end{theorem}
\begin{proof}
By formulas \eqref{PM} and \eqref{CM}, we have
\[\De(K_\al)=\sum_{\eta\unlhd\al}n_{\al\eta}2^{\ell(\eta)}\De(M_\eta)=\sum_{\eta\unlhd\al}n_{\al\eta}2^{\ell(\eta)}
\sum_{\xi\zeta=\eta}M_\xi\ot M_\zeta=\sum_{\xi\zeta\unlhd\al}n_{\al,\,\xi\zeta}2^{\ell(\xi)}M_\xi\ot 2^{\ell(\zeta)}M_\zeta.\]
Next we follow the idea of proof of Prop. \ref{CF}, but have more cases to check. Fix $\xi,\zeta$ such that $\xi\zeta\unlhd\al$. If the last part of $\xi$ is larger than 1, then by definition there exist unique $\be,\ga\in\mW\mC$ such that $\be|\ga=\al,\,\xi\unlhd\be,\,\zeta\unlhd\ga$ and $n_{\al,\,\xi\zeta}=n_{\be\xi}n_{\ga\zeta}$.

Otherwise, if the last part of $\xi$ is zero, there exist pairs  $\be,\ga\in\mW\mC$ such that $\be|\ga=\be\ga=\al,\,\xi\unlhd\be,\,\zeta\unlhd\ga$ instead.
Suppose that the last interval of zeros in $\xi$ is $0^k$. It joins into a interval $0^j$ in $\xi\zeta$, as a contraction of $0^i$ in $\al$. In particular, we have $0<k\leq j\leq i$. The desired identity \beq\label{c4}
n_{\al,\,\xi\zeta}=\sum_{\be\ga=\al\atop\xi\unlhd\be,\,\zeta\unlhd\ga}n_{\be\xi}n_{\ga\zeta},
\eeq
is due to
\beq\label{c2}
{i-\ep+\nu\choose j-\ep+\nu}=\sum_{s=k}^{i-j+k}{s+\nu-1\choose k+\nu-1}{i-\ep-s\choose j-\ep-k}
=\sum_{p=0}^{i-j}{k+\nu-1+p\choose k+\nu-1}{i-\ep-k-p\choose j-\ep-k},
\eeq
where $\ep=\begin{cases}
1,&\mb{if such }0^i\mb{ in }\al\mb{ is at the tail},\\
0,&\mb{otherwise},
\end{cases},\,
\nu=\begin{cases}
1,&\mb{if such }0^i\mb{ in }\al\mb{ follows behind a part }1,\\
0,&\mb{otherwise},
\end{cases}$
and the index $s$ implies $\be$ has $0^s$ in the end.
It can be proved by setting $r_1=k+\nu,\,r_2=j-\ep-k+1,\,n=i-j$ in \eqref{c3} .

If the last part of $\xi$ is 1, it splits into three subcases.

(1) This part 1 in $\xi$ is obtained by splitting a positive part of $\al$.

(2) This part 1 in $\xi$ is obtained by merging an interval $0^r,1$ in $\al$ for $r>0$.

For subcases (1), (2), there exist unique $\be,\ga\in\mW\mC$ such that $\be|\ga=\al,\,\xi\unlhd\be,\,\zeta\unlhd\ga$ and $n_{\al,\,\xi\zeta}=n_{\be\xi}n_{\ga\zeta}$.

(3) This part 1 in $\xi$ is obtained by merging an interval $1,0^s$ in $\al$ for $s\geq0$. Suppose that such part 1 of $\al$ is followed behind by $0^i,\,i\in\bN$, and the beginning of $\zeta$ is $0^j,\,j\in\bN$. Clearly $j\leq i$.
Then there exist pairs $\be,\ga\in\mW\mC$ such that $\be|\ga=\be\ga=\al,\,\xi\unlhd\be,\,\zeta\unlhd\ga$, and we obtain \eqref{c4} by the identity
\[{i-\ep+1\choose j-\ep+1}=\sum_{s=0}^{i-j}{i-s-\ep \choose j-\ep}=\sum_{p=j}^i{p-\ep \choose j-\ep},\]
where $\ep$ is as above and the index $s$ indicates $\be$ has $0^s$ in the end. In summary,
\[\De(K_\al)=\sum_{\xi\zeta\unlhd\al}n_{\al,\,\xi\zeta}2^{\ell(\xi)}M_\xi\ot 2^{\ell(\zeta)}M_\zeta
=\sum_{\be|\ga=\al}\sum_{\xi\unlhd\be}n_{\be\xi}2^{\ell(\xi)}M_\xi\ot \sum_{\zeta\leq\ga}n_{\ga\zeta}2^{\ell(\zeta)}M_\zeta
=\sum_{\be|\ga=\al}K_\be\ot K_\ga.\]

For the formula of $S$ on the WPFQFs, we only need to slightly modify the proof in Prop. \ref{CF}, since $K_\al$ has the same comultiplication rule as $F_\al$. We use the previous sign-reversing involution $\iota$ on $\Om_\al$ to obtain that
\[\begin{split}S(K_\al)&=\sum_{k\geq1}(-1)^k\sum_{\al_1,\dots,\al_k\neq\emptyset\atop
\al_1|\cdots|\al_k=\al}K_{\al_1}\cdots K_{\al_k}=\sum_{k\geq1}(-1)^k\sum_{T_1,\dots,T_k\neq\emptyset\atop
w_1\cdots w_k=w_T}\La(w_1)\cdots \La(w_k)\\
&=\sum_{(x,v)\in\Om_\al}\mb{sgn}\,(x,v)\La(v)=(-1)^{\|\al\|}\La(P_\al^*).
\end{split}\]
where $P_\al^*$ is the dual of $P_\al$. In order to expand $\La(P_\al^*)$ in terms of $K_\be$'s, we also need to eliminate all the opposite order of consecutive labels under overlines in $P_\al^*$. What makes different is that  $\La(P_{1^r},\emptyset)=\La(P_r,\emptyset)=K_{0^r}$ by formula
\eqref{pk}. Thus for each interval $\ol{u_1\dots u_r}$ in $P_\al^*$, $u_1>\cdots>u_r$, we only need to replace it by $\ol{u_r\dots u_1}$ for the computation of $\Ga(P_\al^*)$, and immediately obtain our desired formula.
\end{proof}

For example, if $\al=(2,0^3,1,0)$, $P_\al=67\ol{123}5\ol{4}$, then $P_\al^*=\ol{4}5\ol{321}76$, and we have
\[\begin{split}
\De(K_{20^310})&=K_{20^310}\ot1+K_{20^31}\ot K_0+K_{20^3}\ot K_{10}+K_{20^2}\ot K_{010}\\
&+K_{20}\ot K_{0^210}+K_{2}\ot K_{0^310}+K_1\ot K_{10^310}+1\ot K_{20^310}.\\
S(K_{20^310})&=(-1)^7\La(\ol{4}5\ol{321}76)=-\La(\ol{4}5\ol{123}76)=K_{010^31^2}.
\end{split}\]

\begin{cor}
WPQSym$^0$ is a Hopf subalgebra of WCQSym$^0$, as
\[\De(K_{0^n})=\sum_{i=0}^n K_{0^i}\ot K_{0^{n-i}},\,\ve(K_{0^n})=\de_{n0}\]
and
\[S(K_{0^n})=(-1)^nK_{0^n},\,n\in\bN.\]
\end{cor}

\subsection{Two Hopf algebra projections from WPQSym}
By \cite[Theorem 3.8]{GTY}, there exists a Hopf algebra projection $\varphi:$WCQSym
$\rw$QSym defined by
\[M_\al\stackrel{\varphi}{\mapsto}\begin{cases}
(-1)^{\ell_0(\al)}M_{\ol{\al}},&\mb{there is no }\be\in\mW\mC\mb{ such that }\al=(0,\be),\\
0,&\mb{otherwise}.
\end{cases}\]

By \eqref{FM} and the definition of $\varphi$, we have
\[\varphi(F_\al)=\sum_{\be\leq\al}c_{\al\be}\varphi(M_\be)=
\sum_{\be\leq\al,\,j_1=0}c_{\al\be}(-1)^{\ell_0(\be)}M_{\ol{\be}}
=\sum_{\ga\leq\ol{\al}}\lb\sum_{\be\leq\al,\,\ol{\be}=\ga,\,j_1=0}c_{\al\be}(-1)^{\ell_0(\be)}\rb M_\ga.\]
Since $c_{\al\be}={i_1\choose j_1}\cdots{i_k \choose j_k}{i_{k+1}-1\choose j_{k+1}-1}$, for any $\ga\leq\ol{\al}$,
\[\sum_{\be\leq\al,\,\ol{\be}=\ga,\,j_1=0}c_{\al\be}(-1)^{\ell_0(\be)}
=\prod_{s=2}^k\lb\sum_{j_s=0}^{i_s}(-1)^{j_s}{i_s\choose j_s}\rb\sum_{j_{k+1}=0}^{i_{k+1}}(-1)^{j_{k+1}}{i_{k+1}-1\choose j_{k+1}-1}
=\prod_{s=2}^k\de_{i_s,0}(\de_{i_{k+1},0}-\de_{i_{k+1},1}).\]
Hence, we obtain that
\[\varphi(F_\al)=\begin{cases}
(-1)^jF_{\al_0},&\mb{if }\al=(0^i,\al_0,0^j)\mb{ for some }i\in\bN,\,j\in\{0,1\}\mb{ and }\al_0\in\mC\setminus\{\emptyset\},\\
\de_{\al,\emptyset},&\mb{otherwise}.
\end{cases}\]

According to Theorem \ref{str}, we can define a linear map $\rho:$WPQSym$\rw$PQSym by
\[\rho(K_\al)=\begin{cases}
(-1)^j2^{2-\de_{i0}-\de_{j0}}K_{\al_0},&\mb{if }\al=(0^i,\al_0,0^j)\mb{ for some }i,j\in\bN\mb{ and }\al_0\in\mC\setminus\{\emptyset\},\\
\de_{\al,\emptyset},&\mb{otherwise}.
\end{cases}\]
\begin{theorem}
The map $\rho$ is a Hopf algebra projection from WPQSym onto PQSym.
\end{theorem}
\begin{proof}
It is clear that $\rho$ is a linear projection such that $\rho|_{\bs{PQSym}}=\mb{id}_{\bs{PQSym}}$. We first prove that it is an algebra homomorphism.
Denote $\mW\mC_0$ the subset of $\mW\mC$ consisting of those $\al$ of the form $(0^i,\al_0,0^j)$ for some $i,j\in\bN$ and $\al_0\in\mC\setminus\{\emptyset\}$. Assume that $\al\in\mW\mC\backslash(\mW\mC_0\cup\{\emptyset\})$, then there exist zeros lying between two positive parts of $\al$. For any $\be\in\mW\mC$, by the multiplication rule of WPFQFs, we have $\ga\in\mW\mC\backslash(\mW\mC_0\cup\{\emptyset\})$ for any term $K_\ga$ appeared in the expansion of $K_\al\cdot K_\be$. Hence, $\rho(K_\al\cdot K_\be)=\rho(K_\al)\cdot\rho(K_\be)=0$ in this case.

Otherwise, given $\al,\be\in\mW\mC_0$, we can set $\al=(0^i,\al_0,0^j),\,\be=(0^k,\be_0,0^l)$ for some $i,j,k,l\in\bN$ and $\al_0,\be_0\in\mC\setminus\{\emptyset\}$,
then by \eqref{muk} and the definition of $\rho$,
\[\rho(K_\al\cdot K_\be)=\sum_{u\in(P_\al\sim\be)\shuffle(P_\be\smile\al)}\rho(\La(u))
=\sum_{x\in P_{0^i}\shuffle(P_{0^k}\smile0^{i+j})\atop y\in P_{0^j}\shuffle(P_{0^l}\smile0^{i+j+k})}\sum_{w\in (P_{\al_0}\sim0^{k+l})\shuffle(P_{\be_0}\smile0^{k+l+\|\al\|})}\rho(\La(xwy)).\]
Now substituting \eqref{mukz} into it, we have
\[\rho(K_\al\cdot K_\be)=g_{\al\be}\sum_{w\in P_{\al_0}\shuffle(P_{\be_0}\smile\al_0)}\La(w)=g_{\al\be}K_{\al_0}\cdot K_{\be_0},\]
where the coefficient
\begin{align*}
&g_{\al\be}=\lb\sum_{s=0}^i(-1)^s {i+k-s \choose i}{i \choose s}\dfrac{i+k-2s}{i+k-s}2^{1-\de_{i+k-2s,0}}\rb\\
&\quad\times\lb\sum_{t=0}^j(-1)^t {j+l-t \choose j}{j \choose t}\dfrac{j+l-2t}{j+l-t}
(-1)^{j+l-2t}2^{1-\de_{j+l-2t,0}}\rb.
\end{align*}
Next we prove that $g_{\al\be}=(-1)^{j+l}2^{4-\de_{i0}-\de_{j0}-\de_{k0}-\de_{l0}}$, thus
$\rho(K_\al\cdot K_\be)=\rho(K_\al)\cdot\rho(K_\be)$ for any $\al,\be\in\mW\mC$. By the similarity of those two factors in $g_{\al\be}$, it is certainly enough to show that
\[\sum_{s=0}^i(-1)^s {i+k-s \choose i}{i \choose s}\dfrac{i+k-2s}{i+k-s}2^{1-\de_{i+k-2s,0}}=2^{2-\de_{i0}-\de_{k0}}.\]
Without loss of generality, suppose that $i\leq k$. We prove it by induction on $i$. When $i=0$ it is clear.
For fixed $i<k$, assume the identity holds. Then we check the case for $i+1\leq k$.
\begin{align*}
\sum_{s=0}^{i+1}&(-1)^s {i+1+k-s \choose i+1}{i+1 \choose s}\dfrac{i+1+k-2s}{i+1+k-s}2^{1-\de_{i+1+k-2s,0}}\\
&=2\sum_{s=0}^{i+1}(-1)^s {i+1+k-s \choose i+1}\lb{i \choose s}+{i \choose s-1}\rb\dfrac{i+1+k-2s}{i+1+k-s}\\
&=2\sum_{s=0}^i(-1)^s {i+k-s \choose i}{i \choose s}\dfrac{(i+1)(i+k-2s)+k}{(i+1)(i+k-s)}\\
&=2\lb2^{1-\de_{i0}}+\dfrac{1}{i+1}\sum_{s=0}^i(-1)^s {i+k-1-s \choose k-1}{k \choose s}\rb
=2^{2-\de_{i0}}+2\de_{i0}=4,
\end{align*}
where we use the induction hypothesis
\[\sum_{s=0}^i(-1)^s {i+k-s \choose i}{i \choose s}\dfrac{i+k-2s}{i+k-s}=2^{1-\de_{i0}},\,\mb{as }i<k,\]
and the identity
\[f_{ik}:=\sum_{s=0}^i(-1)^s {i+k-1-s \choose k-1}{k \choose s}=\de_{i0},\,i\geq0,k>0.\]
Such identity can be confirmed as follows. It is clear that $f_{0k}=1$. For any $i>0$,
\begin{align*}
&f_{ik}=\sum_{s=0}^i(-1)^s\left[{i+k-2-s \choose k-1}{k \choose s}
+{i+k-2-s \choose k-2}\lb{k-1 \choose s}+{k-1\choose s-1}\rb\right]\\
&\quad=\sum_{s=0}^i(-1)^s\left[{i+k-2-s \choose k-1}{k \choose s}
+\lb{i+k-2-s \choose k-2}-{i+k-3-s \choose k-2}\rb{k-1 \choose s}\right]\\
&\quad=f_{i-1,k}+f_{i,k-1}-f_{i-1,k-1}.
\end{align*}
It means that $f_{ik}-f_{i,k-1}=f_{i-1,k}-f_{i-1,k-1}=\cdots=f_{0,k}-f_{0,k-1}=0$, and thus
$f_{ik}=f_{i,k-1}=\cdots=f_{i1}=0$ for any $i,k>0$. Eventually, we finish the induction proof that $\rho$ is an algebra map.

Next we need to show that $\rho$ is a coalgebra map.
First assume that $\al=(0^i,\al_0,0^j)\in\mW\mC_0$, then by Theorem \ref{hk} and the definition of $\rho$,
\begin{align*}
&\De\circ\rho(K_\al)=(-1)^j2^{2-\de_{i0}-\de_{j0}}\De(K_{\al_0})=(-1)^j2^{2-\de_{i0}-\de_{j0}}\sum_{\be_0|\ga_0=\al_0}
K_{\be_0}\ot K_{\ga_0}\\
&\quad=
\rho(K_\al)\ot1+1\ot\rho(K_\al)+\sum_{\be_0|\ga_0=\al_0\atop
\be_0,\ga_0\neq\emptyset}\rho(K_{(0^i,\,\be_0)})\ot\rho(K_{(\ga_0,\,0^j)})\\
&\quad=\sum_{\be|\ga=\al}\rho(K_\be)\ot\rho(K_\ga)=(\rho\ot\rho)\circ\De(K_\al).
\end{align*}
Otherwise, if $\al\in\mW\mC\backslash(\mW\mC_0\cup\{\emptyset\})$, we only have to deal with the case when $\al=(a,0^i,b),\,a,b,i\in\bP$. For general $\al$, it can be done by the same argument.
\begin{align*}
&(\rho\ot\rho)\circ\De(K_\al)=\sum_{\be|\ga=\al}\rho(K_\be)\ot\rho(K_\ga)
=\sum_{s=0}^i\rho(K_{(a,0^s)})\ot\rho(K_{(0^{i-s},b)})\\
&\quad=\lb\sum_{s=0}^i(-1)^s2^{2-\de_{s0}-\de_{i-s,0}}\rb K_a\ot K_b
=\lb2(1+(-1)^i)+4\sum_{s=1}^{i-1}(-1)^s\rb K_a\ot K_b
=0=\De\circ\rho(K_\al).
\end{align*}
Hence, we have proven that $\rho$ is a coalgebra map. Moveover, if $\al=(0^i,\al_0,0^j)\in\mW\mC_0$,
\begin{align*}
&\rho\circ S(K_\al)=(-1)^{\|\al\|}\rho(K_{\al^t})=
(-1)^{i+j+|\al_0|}\rho(K_{(0^j,\al_0^t,0^i)})\\
&\quad=(-1)^{j+|\al_0|}2^{2-\de_{i0}-\de_{j0}}K_{\al_0^t}
=(-1)^j2^{2-\de_{i0}-\de_{j0}}S(K_{\al_0})=S\circ\rho(K_\al).
\end{align*}
Other cases for $\al$ are easy to check, and thus $\rho\circ S=S\circ\rho$. Finally we show that $\rho$ is a Hopf algebra map.
\end{proof}

Fix a scalar $b\in\bQ$ and define a linear transformation $\phi_b$ of WCQSym$^0_\bQ$ as
\[\phi_b(F_\emptyset)=K_\emptyset=1,\,\phi_b(F_{0^n})=\sum_{i=1}^nb_{ni}K_{0^i},\,n\in\bP,\]
where the coefficients $b_{ni}\in\bQ$ are defined recursively by
\beq\label{re} b_{11}=b,\,(n+1)b_{n+1,j}=nb_{nj}+jb(b_{n,j-1}-b_{n,j+1}),\,j=1,\dots,n+1,\eeq
with the convention that $b_{n0}=\de_{n0},\,b_{n,n+1}=b_{n,n+2}=0$ for any $n\in\bP$. In particular, we have
\[b_{nn}=b^n,b_{n+1,n}=\dfrac{n}{2}b^n,b_{n+2,n}=\dfrac{1}{24}((3 n^2+5n)b^n-8nb^{n+2}).\]
Actually fix $k\in\bP$, the numbers $b_{n+k,n},\,n\in\bP$ can be computed simultaneously, as we let $e_{n+k,n}:=\prod_{i=1}^k(n+i)b_{n+k,n}$, then
\[e_{n+k,n}-be_{n+k-1,n-1}=(n+k-1)e_{n+k-1,n}-n(n+1)be_{n+k-1,n+1}.\]
When $b_{n+i,n},\,n\in\bP,i<k$ has been calculated, it gives the formula of $b_{n+k,n},\,n\in\bP$.

\begin{prop}\label{hoz}
For any $b\neq0$, the map $\phi_b$ is a Hopf algebra automorphism of WCQSym$^0_\bQ$.
\end{prop}
\begin{proof}
First we show that $\phi_b$ is an algebra automorphism of WCQSym$^0_\bQ$. Indeed,
\[\phi_b(F_0\cdot F_n)=(n+1)\phi_b(F_{0^{n+1}})-n\phi_b(F_{0^n})
=\sum_{j=1}^{n+1}((n+1)b_{n+1,j}-nb_{nj})K_{0^j}.\]
On the other hand,
\[\begin{split}
\phi_b(F_0)\cdot \phi_b(F_{0^n})&=b_{11}K_0\cdot\sum_{i=1}^nb_{ni}K_{0^i}
=b_{11}\sum_{i=1}^nb_{ni}((i+1)K_{0^{i+1}}-(i-1)K_{0^{i-1}})\\
&=b_{11}\lb\sum_{j=2}^{n+1}jb_{n,j-1}K_{0^j}-\sum_{j=1}^{n-1}jb_{n,j+1}K_{0^j}\rb
=b_{11}\lb\sum_{j=1}^{n+1}j(b_{n,j-1}-b_{n,j+1})K_{0^j}\rb.
\end{split}\]
Hence, $\phi_b(F_0\cdot F_{0^n})=\phi_b(F_0)\cdot \phi_b(F_{0^n})$ is equivalent to the relation \eqref{re} with $b_{11}=b$. Next we show that
such special case is sufficient to guarantee that $\phi_b$ is an algebra endomorphism.
Actually for any $m,n\in\bP$, we prove by induction on $m$ that
\[\phi_b(F_{0^m}\cdot F_{0^n})=\phi_b(F_{0^m})\cdot \phi_b(F_{0^n}).\]
The case when $m=1$ is clear. Now assume that it is true for any $m\leq r$, then
\[\begin{split}
\phi_b&((F_0\cdot F_{0^r})\cdot F_{0^n})=
\phi_b(((r+1)F_{0^{r+1}}-rF_{0^r})\cdot F_{0^n})\\
&=(r+1)\phi_b(F_{0^{r+1}}\cdot F_{0^n})-r\phi_b(F_{0^r}\cdot F_{0^n})
=(r+1)\phi_b(F_{0^{r+1}}\cdot F_{0^n})-r\phi_b(F_{0^r})\cdot\phi_b(F_{0^n}).
\end{split}\]
Meanwhile,
\[\begin{split}
\phi_b(F_0&\cdot(F_{0^r}\cdot F_{0^n}))=
\phi_b(F_0)\cdot\phi_b(F_{0^r}\cdot F_{0^n})
=\phi_b(F_0)\cdot(\phi_b(F_{0^r})\cdot\phi_b(F_{0^n}))\\
&=(\phi_b(F_0)\cdot\phi_b(F_{0^r}))\cdot\phi_b(F_{0^n})=
((r+1)\phi_b(F_{0^{r+1}})-r\phi_b(F_{0^r}))\cdot\phi_b(F_{0^n})\\
&=(r+1)\phi_b(F_{0^{r+1}})\cdot\phi_b(F_{0^n})-r\phi_b(F_{0^r})\cdot\phi_b(F_{0^n}).
\end{split}\]
By the associativity of multiplication, we obtain that $\phi_b(F_{0^{r+1}}\cdot F_{0^n})=\phi_b(F_{0^{r+1}})\cdot \phi_b(F_{0^n})$. By definition if $b\neq0$, the matrix of $\phi_b$ with respect to the two bases $\{F_{0^n}\}_{n\in\bN}$ and $\{K_{0^n}\}_{n\in\bN}$ is upper triangular with nonzero diagonal entries $b_{nn}=b^n$, thus $\phi_b$ becomes an algebra automorphism of WCQSym$^0_\bQ$.

Next we prove that $\phi_b$ is a coalgebra map. First note that $\ve\circ\phi_b=\ve$ is obvious. Then we prove by induction on $n\in\bP$ that
\[\De\circ\phi_b(F_{0^n})=(\phi_b\ot\phi_b)\circ\De(F_{0^n}).\]
The case when $n=1$ is clear. Assume that it is true for any $n\leq r$, then
\[\De\circ\phi_b(F_0\cdot F_{0^r})=\De\circ\phi_b((r+1)F_{0^{r+1}}-rF_{0^r})
=(r+1)\De\circ\phi_b(F_{0^{r+1}})-r(\phi_b\ot\phi_b)\circ\De(F_{0^r}).\]
On the other hand, by induction hypothesis we have
\begin{align*}
&\De\circ\phi_b(F_0\cdot F_{0^r})
=\De(\phi_b(F_0)\cdot\phi_b(F_{0^r}))=\De\circ\phi_b(F_0)\cdot\De\circ\phi_b(F_{0^r})\\
&\quad=(\phi_b\ot\phi_b)(F_0\ot1+1\ot F_0)\cdot(\phi_b\ot\phi_b)\lb\sum_{i=0}^rF_{0^i}\ot F_{0^{r-i}}\rb\\
&\quad=(\phi_b\ot\phi_b)\lb(F_0\ot1+1\ot F_0)\cdot\sum_{i=0}^rF_{0^i}\ot F_{0^{r-i}}\rb\\
&\quad=(\phi_b\ot\phi_b)\lb\sum_{i=0}^r((i+1)F_{0^{i+1}}-iF_{0^i})\ot F_{0^{r-i}}+F_{0^i}\ot((r+1-i)F_{0^{r+1-i}}-(r-i)F_{0^{r-i}})\rb\\
&\quad=(r+1)(\phi_b\ot\phi_b)\lb\sum_{i=0}^{r+1}F_{0^i}\ot F_{0^{r+1-i}}\rb-r(\phi_b\ot\phi_b)\lb \sum_{i=0}^rF_{0^i}\ot F_{0^{r-i}}\rb\\
&\quad=(r+1)(\phi_b\ot\phi_b)\circ\De(F_{0^{r+1}})-r(\phi_b\ot\phi_b)\circ\De(F_{0^r}).
\end{align*}
Hence, $\De\circ\phi_b(F_{0^{r+1}})=(\phi_b\ot\phi_b)\circ\De(F_{0^{r+1}})$, and by induction $\phi_b$ is a coalgebra map. Moreover, for $\phi_b$ is an automorphism, we immediately have $S\circ\phi_b=\phi_b\circ S$, and thus finish the proof.
\end{proof}

\begin{rem}
In particular, we have $\phi_{1/2}=$id. By relations \eqref{ik}, \eqref{re}, one can easily check that
\[\phi_{1/2}(F_{0^n})=2^{-n}\sum_{i=1}^n{n-1\choose i-1}K_{0^n}=F_{0^n}\]
for any $n\in\bP$, thus $\phi_{1/2}=$id.
\end{rem}

Now it is natural to ask if $\phi_b$ can be extended to be a Hopf algebra surjection from WCQSym$_\bQ$ to WPQSym$_\bQ$, as an extension of Stembridge's descent-to-peak map $\te$. In fact, it is nearly guaranteed by Prop. \ref{hoz}. Here for simplicity, we only extend $\phi_{1/2}$, the identity on WCQSym$^0_\bQ$, to the whole algebra WCQSym$_\bQ$.

As $\{F_\al\}_{\al\in\mC}$ is a basis of WCQSym, we can define a linear map
\[\Te:\mb{WCQSym}_\bQ\rw\mb{WPQSym}_\bQ,\,F_\al\mapsto\sum_{\ol{\be}=\ol{\al}}u_{\al\be}K_\be\]
for $\al=(0^{i_1},s_1,\dots,0^{i_k},s_k,0^{i_{k+1}})$ with all $i_p\in\bN$ and $s_q\in\bP$, where $u_{\al\be}=2^{-\ell_0(\al)}\prod\limits_{p=1}^{k+1}{i_p-1\choose j_p-1}$ for  $\be=(0^{j_1},s_1,\dots,0^{j_k},s_k,0^{j_{k+1}})$. In particular, $u_{\al\be}>0$ only when $1\leq j_p\leq i_p$ or $j_p=i_p=0$ for any $p\in[k+1]$.

\begin{theorem}
The map $\Te$ is a Hopf algebra surjection from WCQSym$_\bQ$ onto WPQSym$_\bQ$. Moreover, the restriction of $\Te$ on QSym$_\bQ$ is Stembridge's descent-to-peak map $\te$, and the commutation equality $\rho\circ\Te=\te\circ\varphi$ holds.
\end{theorem}
\begin{proof}
By definition, $\Te$ is surjective, as it is the extension of $\phi_{1/2}=$id$_{\bs{WCQSym}^0_\bQ}$ mapping onto  WPQSym$_\bQ$, and for any $\al\in\mC$, $\Te(F_\al)=K_\al$, thus $\Te|_{\bs{QSym}}=\te$.

Next we show that $\Te$ is an algebra map. For any nonempty $\al,\be\in\mW\mC$,
\[\Te(F_\al)\cdot\Te(F_\be)=\sum_{\ga,\eta}u_{\al\ga}u_{\be\eta}K_\ga\cdot K_\eta=\sum_{\ga,\eta}u_{\al\ga}u_{\be\eta}\sum_{w\in(P_\ga\sim\eta)\shuffle(P_\eta\smile\ga)}\La(w),\]
and
\[\Te(F_\al\cdot F_\be)=\sum_{u\in(P_\al\sim\be)\shuffle(P_\be\smile\al)}\Te(\Ga(u)).\]

To show that these two kinds of expansions coincide, we only need to deal with the following three special cases:

(1) $\al=0^m,\,\be=0^n$. This case is clear, since $\Te|_{\bs{WCQSym}^0_\bQ}=\phi_{1/2}=\mb{id}_{\bs{WCQSym}^0_\bQ}$ is an algebra map.

(2) $\al,\,\be\in\mC$. This case is also clear, since $\Te|_{\bs{QSym}}=\te$ is an algebra map.

(3) $\al=0^m,\,\be=r\in\bP$. This is a new case only happened in WCQSym, and needs
further illustration.
\begin{align*}
&\Te(F_{0^m})\cdot\Te(F_r)=2^{-m}\sum_{j=1}^m{m-1\choose j-1}K_{0^j}\cdot K_r\\
&\quad=2^{-m}\sum_{j=1}^m{m-1\choose j-1}\sum_{s_1,\dots,s_k\in\bP\atop s_1+\cdots+s_k=r}\sum_{j_1,j_{k+1}\in\bN,\,j_2,\dots,j_k\in\bP\atop j_1+\cdots+j_{k+1}=j}K_{(0^{j_1},s_1,\dots,0^{j_k},s_k,0^{j_{k+1}})},\\
&\Te(F_{0^m}\cdot F_r)=
\sum_{s_1,\dots,s_k\in\bP\atop s_1+\cdots+s_k=r}\sum_{i_1,i_{k+1}\in\bN,\,i_2,\dots,i_k\in\bP\atop i_1+\cdots+i_{k+1}=m}\Te(F_{(0^{i_1},s_1,\dots,0^{i_k},s_k,0^{i_{k+1}})})\\
&\quad=2^{-m}\sum_{s_1,\dots,s_k\in\bP\atop s_1+\cdots+s_k=r}\sum_{i_1,i_{k+1}\in\bN,\,i_2,\dots,i_k\in\bP\atop i_1+\cdots+i_{k+1}=m}\sum_{0\leq j_p\leq i_p\atop
p=1,\dots,k+1}\prod_{p=1}^{k+1}{i_p-1\choose j_p-1}
K_{(0^{j_1},s_1,\dots,0^{j_k},s_k,0^{j_{k+1}})}\\
&\quad=2^{-m}\sum_{j=1}^m\sum_{s_1,\dots,s_k\in\bP\atop s_1+\cdots+s_k=r}
\sum_{j_1,j_{k+1}\in\bN,\,j_2,\dots,j_k\in\bP\atop j_1+\cdots+j_{k+1}=j}
\lb\sum_{i_p\geq j_p,\,p\in[k+1]\atop i_1+\cdots+i_{k+1}=m}\prod_{p=1}^{k+1}{i_p-1\choose j_p-1}\rb
K_{(0^{j_1},s_1,\dots,0^{j_k},s_k,0^{j_{k+1}})}.
\end{align*}
Hence, it reduces to the proof of the following identity
\[{m-1\choose j-1}=\sum_{i_p\geq j_p,\,p\in[k]\atop i_1+\cdots+i_k=m}\prod_{p=1}^k{i_p-1\choose j_p-1}\]
for any given $j\in[m]$ and $j_1,\dots,j_k\in\bP$ such that $j_1+\cdots+j_k=j$ with $k\leq j$. It is due to the generalization of \eqref{c3} with $k$ parts,
\[{r_1+\cdots+r_k+n-1\choose r_1+\cdots+r_k-1}=\sum_{n_p\in\bN,\,p\in[k]
\atop n_1+\cdots+n_k=n}\prod_{p=1}^k{r_p+n_p-1\choose r_p-1},\,r_1,\dots,r_k\in\bP,n\in\bN,\]
by taking $r_p=j_p,n_p=i_p-j_p,\,p=1,\dots,k$ and $n=m-j$.

On the other hand, for fixed $\be,\ga,\eta$ satisfying $\ol{\be}=\ol{\al},\,u_{\al\be}>0$ and $\be=\ga|\eta$, if
both the last part of $\ga$ and the first one of $\eta$ are positive, then
by definition there exist unique $\xi,\zeta$ such that $\xi|\zeta=\al,\,\ol{\xi}=\ol{\ga},\,\ol{\zeta}=\ol{\eta}$ and $u_{\al\be}=u_{\xi\ga}u_{\zeta\eta}$.
Otherwise, if either the last part of $\ga$ or the first one of $\eta$ is 0, we have the identity
\[u_{\al\be}=\sum_{\xi|\zeta=\al}\sum_{\ol{\xi}=\ol{\ga}\atop\ol{\zeta}=\ol{\eta}}u_{\xi\ga}u_{\zeta\eta}.\]
Indeed, in this case a interval $0^j$ of $\be$ as a contraction of $0^i$ in $\al$ is a combination of the last interval $0^k$ in $\ga$ and the first $0^{j-k}$ in $\eta$ with $0\leq k\leq j\leq i$. Then the above identity is due to \eqref{c3} by taking $r_1=k,\,r_2=j-k$ and $n=i-j$. Consequently,
\begin{align*}
&\De\circ\Te(F_\al)=\sum_{\ol{\be}=\ol{\al}}u_{\al\be}\De(K_\be)
=\sum_{\ol{\be}=\ol{\al}}u_{\al\be}\sum_{\ga|\eta=\be}K_\ga\ot K_\eta\\
&=\sum_{\xi|\zeta=\al}\lb\sum_{\ol{\xi}=\ol{\ga}}u_{\xi\ga}K_\ga\rb\ot\lb\sum_{\ol{\zeta}=\ol{\eta}}u_{\zeta\eta}K_\eta\rb
=\sum_{\xi|\zeta=\al}\Te(K_\xi)\ot\Te(K_\zeta)=(\Te\ot\Te)\circ\De(F_\al).
\end{align*}
Hence, $\Te$ is also a coalgebra map. Meanwhile, if $\ol{\be}=\ol{\al}$ and $u_{\al\be}>0$, it also implies that $\ol{\be^t}=\ol{\al^t}$ and $u_{\al^t,\,\be^t}>0$, and vice versa. In this case $u_{\al\be}=u_{\al^t,\,\be^t}$. Thus we have
\[S\circ\Te(F_\al)=\sum_{\ol{\be}=\ol{\al}}u_{\al\be}S(K_\be)
=\sum_{\ol{\be}=\ol{\al}}(-1)^{\|\be\|}u_{\al\be}K_{\be^t}
=\sum_{\ol{\be}=\ol{\al^t}}(-1)^{\|\be\|}u_{\al^t\be}K_\be.\]
On the other hand,
\[\Te\circ S(F_\al)=\sum_{\ol{\ga}=\ol{\al^t}}(-1)^{\|\ga\|}d_{\al^t,\ga}\Te(F_\ga)
=\sum_{\ol{\ga}=\ol{\al^t}}\sum_{\ol{\be}=\ol{\ga}}(-1)^{\|\ga\|}d_{\al^t,\ga}u_{\ga\be}K_\be
=\sum_{\ol{\be}=\ol{\al^t}}\sum_{\ol{\ga}=\ol{\be}}(-1)^{\|\ga\|}d_{\al^t,\ga}u_{\ga\be}K_\be.\]
Therefore, it is enough to prove that $(-1)^{\|\be\|}u_{\al^t\be}=\sum_{\ol{\ga}=\ol{\be}}(-1)^{\|\ga\|}d_{\al^t,\ga}u_{\ga\be}$ for any $\al,\be$ such that $\ol{\be}=\ol{\al^t}$. Assume that
$\al^t=(0^{i_1},s_1,\dots,0^{i_k},s_k,0^{i_{k+1}})$ with all $i_p\in\bN$ and $s_q\in\bP$, it is nontrivial only when
$\be=(0^{j_1},s_1,\dots,0^{j_k},s_k,0^{j_{k+1}})$ with $1\leq j_p\leq i_p$ or $j_p=i_p=0$ for any $p\in[k+1]$.
Thus we have
\begin{align*}
&\sum_{\ol{\ga}=\ol{\be}}(-1)^{\|\ga\|}d_{\al^t,\ga}u_{\ga\be}=(-1)^{|\be|}\sum_{j_p\leq r_p\leq i_p\atop p=1,\dots,k+1}\prod\limits_{p=1}^{k+1}{i_p-1\choose r_p-1}
{r_p-1\choose j_p-1}(-1)^{r_p}2^{-r_p}\\
&\quad=(-1)^{|\be|}\prod\limits_{p=1}^{k+1}{i_p-1\choose j_p-1}\lb\sum_{j_p\leq r_p\leq i_p}{i_p-j_p\choose r_p-j_p}(-1)^{r_p}2^{-r_p}\rb\\
&\quad=(-1)^{\|\be\|}\prod\limits_{p=1}^{k+1}{i_p-1\choose j_p-1}2^{-i_p}(2-1)^{i_p-j_p}
=(-1)^{\|\be\|}u_{\al^t\be},
\end{align*}
which ensures that $S\circ\Te=\Te\circ S$. As a result, $\Te$ is a Hopf algebra map. Moreover, if $\al=(0^i,\al_0,0^j)$ for some $i,j\in\bN$ and $\al_0\in\mC\setminus\{\emptyset\}$, then
\begin{align*}
&\rho\circ\Te(F_\al)=\sum_{\ol{\be}=\ol{\al}}u_{\al\be}\rho(K_\be)=2^{-i-j}\sum_{0\leq k\leq i\atop
0\leq l\leq j}{i-1\choose k-1}{j-1\choose l-1}
(-1)^l2^{2-\de_{k0}-\de_{l0}}K_{\al_0}\\
&=2^{-i-j}\cdot  2^i\cdot(\de_{j0}-2\de_{j1})K_{\al_0}=(\de_{j0}-\de_{j1})\te(F_{\al_0})
=\te\circ\varphi(F_\al).
\end{align*}
For other cases of $\al$, we have $\rho\circ\Te(F_\al)=\te\circ\varphi(F_\al)=\de_{\al,\emptyset}$. Finally we finish the proof.
\end{proof}
As in Theorem \ref{str}, we know that for any $\al\in\mW\mC$, there exists a unique $\be\in\mP\mC$ such that $K_\al=K_\be$ by Lemma \ref{eq}. Hence, it induces a map $\tau:\mW\mC\rw\mP\mC$ such that $K_{\tau(\al)}=K_\al$. In particular, Ker\,$\Te=\mb{span}_\bQ\{F_\al-F_{\tau(\al)}\}_{\al\in\mW\mC}$.

\section{Rota-Baxter algebras and weak peak quasisymmetric functions}
In this section we study Rota-Baxter algebra structures on WPQSym. First recall that
\begin{defn}
Given a commutative ring  $k$ and $\la\in k$,
an algebra $A$ over $k$ is called a \textit{Rota-Baxter algebra} of weight $\la$, if it is equipped with a $k$-linear operator $P:A\rw A$ satisfying the Rota-Baxter identity
\beq\label{rb}
P(x)P(y)=P(xP(y))+P(P(x)y)+\la P(xy),\,\forall x,y\in A.
\eeq
Such a operator $P$ is called a \textit{Rota-Baxter operator} of weight $\la$.
\end{defn}

There exists a natural Rota-Baxter algebra structure relative to WCQSym by \cite[Theorem 4.5]{GTY}.
Denote $x^{\ot \al}:=x^{\al_1}\ot\cdots\ot x^{\al_r}$ for $\al=(\al_1,\dots,\al_r)\in\bN^r$.
Let
\[\shuffle^+_\la(x):=\bigoplus\limits_{n\geq0}\bZ[x]^{\ot n}=\bigoplus\limits_{\al\in\mW\mC}\bZ x^{\ot \al},\]
which has the \textit{mixable shuffle product} $*_\la$ defined recursively by
\[1*_\la a=a*_\la 1=a,\,a*_\la b=a_1\ot(a'*_\la b)+b_1\ot(a*_\la b')+\la a_1b_1\ot(a'*_\la b')\]
for any $a=a_1\ot a'$, $b=b_1\ot b'\in\shuffle^+_\la(x)$.
By convention, $\bZ[x]^{\ot0}=\bZ$ and $x^{\ot\emptyset}=1\in\bZ$. Let \[\shuffle_\la(x):=\bZ[x]\ot\shuffle^+_\la(x)
=\bigoplus\limits_{n\geq1}\bZ[x]^{\ot n}
=\bigoplus\limits_{\al\in\bN^k,\,k\in\bP}x^{\ot\al},\]
which has the \textit{augmented mixable shuffle product} $\diamond_\la$ defined by
\beq\label{ms}(a_0\ot a)\diamond_\la(b_0\ot b)=a_0b_0\ot(a*_\la b),\eeq
for any $a_0\ot a=a_0\ot(a_1\ot\cdots\ot a_m)$, $b_0\ot b=b_0\ot (b_1\ot\cdots\ot b_n)\in\shuffle_\la(x)$.

By \cite[Theorem 4.1]{EG}, $(\shuffle_\la(x),\diamond_\la)$ is the free commutative unitary Rota-Baxter algebra of weight $\la$ generated by $\bZ[x]$, with the RB operator $P$ defined as
\[P(a_0\ot a)=\begin{cases}
1\ot(a_0\ot a),&a\in\bZ[x]^{\ot n},\,n\geq1,\\
1\ot aa_0,&a\in\bZ[x]^{\ot 0}=\bZ.
\end{cases}\]
for any $a_0\in\bZ[x]$; see also \cite{Guo}. In addition,  $\shuffle^+_\la(x)$ embeds as a Rota-Baxter subalgebra of $\shuffle_\la(x)$,  since there is an algebra isomorphism $\eta:(\shuffle^+_\la(x),*_\la)\rw(1\ot \shuffle^+_\la(x),\diamond_\la)$.
In particular, let $\shuffle^+(x):=\shuffle^+_1(x)$, then it is isomorphic to WCQSym identifing $x^{\ot \al}$ with $M_\al$ and serves as a Rota-Baxter subalgebra of $\shuffle(x):=\shuffle_1(x)$. Hence, if we abuse the notation $P$ for the Rota-Baxter operator on WCQSym, then $P(M_\al)=M_{(0,\al)},\,\forall\al\in\mW\mC$.
Moreover, viewing that $\shuffle(x)$ becomes a tensor product of two Hopf algebras $\bZ[x]$ and $\shuffle^+(x)\,(\cong\mb{WCQSym})$, the authors in \cite{GTY} successively provide a Hopf algebra structure on the free commutative unitary Rota-Baxter algebra $\shuffle(x)$.

From formula \eqref{gm}, we especially have
\begin{align*}
P(F_{0^r})&=\sum_{i=1}^r{r-1\choose i-1}P(M_{0^i})=\sum_{i=1}^r{r\choose i}M_{0^{i+1}}-\sum_{i=1}^r{r-1\choose i}M_{0^{i+1}}\\
&=\sum_{i=1}^{r+1}{r\choose i-1}M_{0^i}-\sum_{i=1}^r{r-1\choose i-1}M_{0^i}=F_{0^{r+1}}-F_{0^r},\,r\in\bP,
\end{align*}
and $P(F_\emptyset)=P(1)=F_0$. In general from formula \eqref{FM}, we obtain that
\[P(F_\al)=F_{(0,\al)}-(1-\de_{\al,\emptyset})F_\al
=\sum_{\be\leq\al}c_{\al\be}M_{(0,\,\be)}\]
for any $\al\in\mW\mC$, since $c_{(0,\al),(0,\,\be)}-c_{\al,(0,\,\be)}=c_{\al\be}$ for any $\be\leq\al$.
However, the action of RB-operator $P$ on $K_\al$'s is much subtler.
\begin{theorem}\label{RBK}
WPQSym$_\bQ$ is a Rota-Baxter subalgebra of WCQSym$_\bQ$, as we have the following formula,
\beq\label{rbk}
P(K_\al)=\sum_{\be\unlhd\al}n_{\al\be}2^{\ell(\be)}M_{(0,\,\be)}
=\begin{cases}
\tfrac{1}{2}(K_{(0,\al)}-K_\al+K_{\al_-}),&\mb{if }i_1=0,\,i_2>0\mb{ and }s_1=1,\\
\tfrac{1}{2}(K_{(0,\al_+)}-K_{\al_+}),&\mb{if }i_1=i_2=0,\,s_1=1, \mb{ but }\al\neq(1),\\
\tfrac{1}{2}(K_{(0,\al)}-(1-\de_{\al,\emptyset})K_\al),&\mb{otherwise},
\end{cases}
\eeq
for any $\al=(0^{i_1},s_1,\dots,0^{i_k},s_k,0^{i_{k+1}})$ with all $i_p\in\bN,\,s_q\in\bP$, where we denote $\al_-:=(s_1,0^{i_2-1},s_2,0^{i_3},\dots,s_k,0^{i_{k+1}})$ and $\al_+=(s_1+s_2,0^{i_3},\dots,s_k,0^{i_{k+1}})$.
\end{theorem}
\begin{proof}
Fix any $\be=(0^{j_1},t_1,\dots,0^{j_r},t_r,0^{j_{r+1}})\unlhd\al$. First note that 
\beq\label{nab}n_{(0,\al),(0,\ga)}-n_{\al,(0,\ga)}=n_{\al\ga}\eeq
for any $\al,\ga\in\mW\mC$, as the first $0$ in $(0,\al)$ may merge with the next part or not when $(0,\al)$ mutate to $(0,\ga)$. If $s_1\geq2$ or $\al=0^{i_1}$,
due to \eqref{PM}, \eqref{nab} and
\beq\label{nab1}
n_{(0,\al),\,\be}=(1-\de_{\al,\emptyset})n_{\al\be},\mb{ if }j_1=0,
\eeq
we have
\begin{align*}
P(K_\al)&=\sum_{\ga\unlhd\al}n_{\al\ga}2^{\ell(\ga)}M_{(0,\ga)}=\sum_{\ga\unlhd\al}\lb n_{(0,\al),(0,\ga)}-n_{\al,(0,\ga)}\rb 2^{\ell(\ga)}M_{(0,\ga)}\\
&=\dfrac{1}{2}\lb\sum_{\ga\unlhd\al}n_{(0,\al),(0,\ga)}2^{\ell(0,\ga)}M_{(0,\ga)}
-\sum_{\ga\unlhd\al}n_{\al,(0,\ga)}2^{\ell(0,\ga)}M_{(0,\ga)}\rb\\
&=\dfrac{1}{2}\lb\sum_{\be\unlhd(0,\al)}n_{(0,\al),\,\be}2^{\ell(\be)}M_\be
-(1-\de_{\al,\emptyset})\sum_{\be\unlhd\al}n_{\al\be}2^{\ell(\be)}M_\be\rb
=\dfrac{1}{2}(K_{(0,\al)}-(1-\de_{\al,\emptyset})K_\al).
\end{align*}

Otherwise, we assume that exactly the first $r$ positive parts of $\al$ are $1$, i.e. $s_1,\dots,s_r=1$. By \eqref{11}, $n_{\al\be}$ in this case has the factor
\[\sum_{\ep_1,\dots,\ep_{r+1}\in\{0,1\}}{i_1-1 \choose j_1-\ep_1}
{i_2-1+\ep_1 \choose j_2-\ep_2+\ep_1}
\cdots{i_r-1+\ep_{r-1} \choose j_r-\ep_r+\ep_{r-1}}
{i_{r+1}-\ep_{r+1}+\ep_r\choose j_{r+1}-\ep_{r+1}+\ep_r}\]
when $r<k$, otherwise $\ep_{r+1}$ can only take 1 if $r=k$.

Now if $i_1>0$, one can check that the identities \eqref{nab}, \eqref{nab1} still work, thus we obtain the desired result as above.
If $i_1=i_2=0$, then $\al$ has $s_1=1$ followed by a positive $s_2$ adjacently, or $\al=(1)$. It is clear when $\al=(1)$. For the case when
$\al$ begins with $s_1=1,s_2>0$, by Lemma \ref{eq} $K_\al=K_{\al_+}$, and we return back to the previous case when $s_1\geq2$.
The case when $i_1=0,\,i_2>0$ is more complicated. For such $\al$, note that
$n_{(0,\al),\,\be}$'s differ a lot between two kinds of $\be$ with $j_1=0$ or $j_1=1$.
If $j_1=0$, $n_{(0,\al),\,\be}$ has the factor
\[\sum_{\ep_2,\dots,\ep_{r+1}\in\{0,1\}}{1-1 \choose 0-0}{i_2-1+0\choose j_2-\ep_2+0}\cdots{i_{r+1}-\ep_{r+1}+\ep_r\choose j_{r+1}-\ep_{r+1}+\ep_r}.\]
Now $i_2>0$ implies that ${i_2-1\choose j_2-\ep_2}={i_2\choose j_2-\ep_2+1}-{i_2-1\choose j_2-\ep_2+1}$ or
${i_2-\ep_2\choose j_2-\ep_2}={i_2-\ep_2+1\choose j_2-\ep_2+1}-{i_2-\ep_2\choose j_2-\ep_2+1}$, and thus we obtain that
\[n_{(0,\al),\,\be}=n_{\al\be}-n_{\al_-,\,\be},\mb{ when }j_1=0.\]
If $j_1=1$, we can write  $\be=(0,\be_0)$. By \eqref{nab} we know that $n_{(0,\al),\,\be}=n_{\al\be_0}$, as $n_{\al\be}=0$.
In conclusion, when $i_1=0,\,i_2>0$, by \eqref{PM} we have
\begin{align*}
P(K_\al)&=\sum_{\be_0\unlhd\al}n_{\al\be_0}2^{\ell(\be_0)}M_{(0,\,\be_0)}
=\dfrac{1}{2}\sum_{\be\unlhd(0,\al),\,j_1=1}n_{(0,\al),\,\be}2^{\ell(\be)}M_\be\\
&=\dfrac{1}{2}\lb\sum_{\be\unlhd(0,\al)}n_{(0,\al),\,\be}2^{\ell(\be)}M_\be
-\sum_{\be\unlhd(0,\al),\,j_1=0}n_{(0,\al),\,\be}2^{\ell(\be)}M_\be\rb\\
&=\dfrac{1}{2}\lb\sum_{\be\unlhd(0,\al)}n_{(0,\al),\,\be}2^{\ell(\be)}M_\be
-\sum_{\be\unlhd\al}n_{\al\be}2^{\ell(\be)}M_\be+\sum_{\be\unlhd\al_-}n_{\al_-,\,\be}2^{\ell(\be)}M_\be\rb\\
&=\tfrac{1}{2}(K_{(0,\al)}-K_\al+K_{\al_-}).\qedhere
\end{align*}
\end{proof}
For example, one can easily check that
\[P(K_{10^r1})=\begin{cases}
\tfrac{1}{2}(K_{010^r1}-K_{10^r1}+K_{10^{r-1}1}),&r\geq1,\\
\tfrac{1}{2}(K_{02}-K_2),&r=0.
\end{cases}\]

\begin{rem}
It is rigid for the map $\phi_b$ in Prop. \ref{hoz} to become a Rota-Baxter algebra automorphism of WCQSym$^0$. Actually it is true only when $b=\tfrac{1}{2}$ and $\phi_{1/2}=$id, since
\[P\circ\phi_b(F_{0^n})=\sum_{i=1}^nb_{ni}P(K_{0^i})=\dfrac{1}{2}\sum_{i=1}^nb_{ni}(K_{0^{i+1}}-K_{0^i})
=\dfrac{1}{2}\lb b^nK_{0^{n+1}}+\sum_{i=1}^n(b_{n,i-1}-b_{ni})K_{0^i}\rb\]
and
\[\phi_b\circ P(F_{0^n})=\phi_b(F_{0^{n+1}}-F_{0^n})
=b^{n+1}K_{0^{n+1}}+\sum_{i=1}^n(b_{n+1,i}+b_{ni})K_{0^i}.\]
\end{rem}

We also note that the map $\Te$ is not a homomorphism of Rota-Baxter algebras, with respect to the Rota-Baxter subalgebra structure of WPQSym$_\bQ$ in WCQSym$_\bQ$. It intrigues us to endow WPQSym$_\bQ$ with another proper Rota-Baxter operator to meet such requirement.
\begin{theorem}\label{rb1}
WPQSym$_\bQ$ has another different Rota-Baxter algebra structure defined by
\[\hat{P}(K_\al)=\dfrac{1}{2}(K_{(0,\al)}-2^{\de_{i_1,0}}(1-\de_{\al,\emptyset})K_\al)\]
for any $\al=(0^{i_1},s_1,\dots,0^{i_k},s_k,0^{i_{k+1}})$ with all $i_p\in\bN,\,s_q\in\bP$, where we denote $\hat{P}$ the Rota-Baxter operator defined here avoiding confusion with the previous one.
Moreover, we have the commutation equality $\hat{P}\circ\Te=\Te\circ P$.
\end{theorem}
\begin{proof}
First note that $\hat{P}$ is well-defined, as it can be defined on the basis proposed in Theorem \ref{str} and then linearly extended onto the whole algebra WPQSym$_\bQ$. Since $\Te$ is a surjective algebra map and $P$ is a Rota-Baxter operator, we only need to prove the commutation equality stated in the theorem to ensure that $\hat{P}$ is a Rota-Baxter operator on WPQSym.

If $\al=\emptyset$, then
$\Te\circ P(F_\emptyset)=\Te(F_0)=\dfrac{1}{2}K_0=\hat{P}(K_\emptyset)=\hat{P}\circ\Te(F_\emptyset)$.
Otherwise, as $\al=(0^{i_1},s_1,\dots,0^{i_k},s_k,0^{i_{k+1}})\in\mW\mC\backslash\{\emptyset\}$,
we first note that
\begin{align*}
&u_{(0,\al),\,\be}-u_{\al\be}=2^{-\ell_0(\al)-1}\lb{i_1\choose j_1-1}-2{i_1-1\choose j_1-1}\rb\prod\limits_{p=2}^{k+1}{i_p-1\choose j_p-1}\\
&\quad=2^{-\ell_0(\al)-1}\lb{i_1-1\choose j_1-2}-2^{\de_{i_1,0}}{i_1-1\choose j_1-1}\rb\prod\limits_{p=2}^{k+1}{i_p-1\choose j_p-1}=\dfrac{1}{2}(u_{\al\be'}-2^{\de_{i_1,0}}u_{\al\be})
\end{align*}
for $\be=(0^{j_1},s_1,\dots,0^{j_k},s_k,0^{j_{k+1}})$, and we denote $\be=(0,\be')$ when $j_1\geq1$. Meanwhile, by definition if $u_{\al\be}>0$, then $\de_{i_1,0}=\de_{j_1,0}$. Hence, we have
\begin{align*}
&\Te\circ P(F_\al)=\Te(F_{(0,\al)}-F_\al)=\sum_{\ol{\be}=\ol{\al}}(u_{(0,\al),\,\be}-u_{\al\be})K_\be=
\dfrac{1}{2}\sum_{\ol{\be}=\ol{\al}}(u_{\al\be'}-2^{\de_{i_1,0}}u_{\al\be})K_\be\\
&\quad=\dfrac{1}{2}\sum_{\ol{\be}=\ol{\al}}u_{\al\be}(K_{(0,\,\be)}-2^{\de_{j_1,0}}K_\be)
=\sum_{\ol{\be}=\ol{\al}}u_{\al\be}\hat{P}(K_\be)=\hat{P}\circ\Te(F_\al). \qedhere
\end{align*}
\end{proof}
By the previous results, it is easy to check that there exists a Hopf algebra projection
\[\pi:\mb{WCQSym}\rw\mb{WCQSym}^0,\,F_\al\mapsto\begin{cases}
F_\al,&\al=0^r\mb{ for some }r\in\bN,\\
0,&\mb{otherwise}.
\end{cases}\]
Also, we note that the restriction of $\pi$ on WPQSym is a Hopf algebra projection onto WPQSym$^0$ fixing
all $K_{0^r}$'s and vanishing other basis elements. We abuse the notation to denote it by $\pi$.
\begin{cor}
WPQSym$^0_\bQ$ is a Rota-Baxter subalgebra of WPQSym$_\bQ$ with RB operator $P$, and a Rota-Baxer quotient algebra of WPQSym$_\bQ$ with RB operator $\hat{P}$ under the projection $\pi$.
\end{cor}

\begin{rem}
Recently, the dual objects of Rota-Baxter algebras, called Rota-Baxter coalgebras, are introduced, and later Rota-Baxter bialgebras are also defined with such two structures compatible; see \cite{JZ,ML}.
For (WCQSym,$P$), one can easily check the following identity
\[(P\ot P)\circ(\De-\mb{id}\ot(\mu\circ\ve))=(\mb{id}\ot P)\circ\ol{\De}\circ P,\]
with $\ol{\De}:=\De-\mb{id}\ot(\mu\circ\ve)-(\mu\circ\ve)\ot\mb{id}$, the reduced coproduct, which is different from the defining condition for Rota-Baxter coalgebras dual to \eqref{rb}. By the commutation equality in Theorem \ref{rb1}, (WPQSym,$\hat{P}$) also obeys the above rule. It is interesting to find a Rota-Baxter coalgebra (and even RB-bialgebra) structure on WCQSym or WPQSym.
\end{rem}

In the end, we finish our discussion about weak composition (resp. peak) quasisymmetric functions with a commutative diagram as a conclusion.
\[\xymatrix@=2em{
&\mb{WPQSym}^0_\bQ\ar@{->}[r]\ar@/^/[d]
&\mb{WCQSym}^0_\bQ\ar@{->}[r]^-{\phi_{1/2}}\ar@/^/[d]
&\mb{WPQSym}^0_\bQ\ar@/^/[d]&\\
\mb{WPQSym}^0_\bQ\ar@{->}[r]\ar@{->}[d]^-{P}
&\mb{WPQSym}_\bQ\ar@{->}[r]\ar@{->}[d]^-{P}\ar@/^/[u]^-{\pi}
&\mb{WCQSym}_\bQ\ar@{->}[r]^-{\Te}\ar@{->}[d]^-{P}\ar@/^/[u]^-{\pi}
&\mb{WPQSym}_\bQ\ar@{->}[r]^-{\pi}\ar@{->}[d]^-{\hat{P}}\ar@/^/[u]^-{\pi}
&\mb{WPQSym}^0_\bQ\ar@{->}[d]^-{\hat{P}}\\
\mb{WPQSym}^0_\bQ\ar@{->}[r]
&\mb{WPQSym}_\bQ\ar@{->}[r]
&\mb{WCQSym}_\bQ\ar@{->}[r]^-{\Te}\ar@/^/[d]^-{\varphi}
&\mb{WPQSym}_\bQ\ar@{->}[r]^-{\pi}\ar@/^/[d]^-{\rho}
&\mb{WPQSym}^0_\bQ\\
&\mb{PQSym}_\bQ\ar@{->}[r]\ar@{->}[u]
&\mb{QSym}_\bQ\ar@{->}[r]^-{\te}\ar@/^/[u]
&\mb{PQSym}_\bQ\ar@/^/[u]&},\]
where all the arrows without labels stuck on represent canonical inclusions.
The whole picture tells us that it is a really meaningful lifting from (peak) quasisymmetric functions to the weak case, in perspective of Hopf algebras also with extraordinary compatible Rota-Baxter algebra structure.

\centerline{\bf Acknowledgments}
This work is supported by the NSFC Grants (No. 11501214, 11771142) and Guangzhou University's 2017 training program for young top-notch personnels.

\bigskip
\bibliographystyle{amsalpha}

\begin{thebibliography}{9999}
\medskip

\bibitem{ABS} M. Aguiar, N. Bergeron, F. Sottile:
\textit{Combinatorial Hopf Algebras and generalized Dehn-Sommerville relations}, Compositio Math. \textbf{142} (2006), 1--30.

\bibitem{AS} M. Aguiar, F. Sottile: \textit{Structure of the Malvenuto-Reutenauer Hopf algebra of permutations}, Adv.
Math. \textbf{191} (2005), 225--275.

\bibitem{AGKO} G. E. Andrews, L. Guo, W. Keigher, K. Ono: \textit{Baxter algebras and Hopf algebras}, Trans. Amer. Math. Soc. \textbf{355} (2003), 4639--4656.

\bibitem{BS} C. Benedetti, B. Sagan: \textit{Antipodes and involutions}, J. Comb. Theory, Ser. A \textbf{148} (2017), 275--315.

\bibitem{BHT} N. Bergeron, F. Hivert, J.-Y. Thibon:
    \textit{The peak algebra and the Hecke-Clifford algebras at $q=0$}, J. Comb. Theory, Ser. A \textbf{107} (2004), 1--19.

\bibitem{BLL} N. Bergeron, T. Lam, H. Li:
\textit{Combinatorial {H}opf algebras and towers of algebra-dimension, quantization and functorality}, Algebr Represent Theor \textbf{15} (2012), 675--696.

\bibitem{BL} N. Bergeron, H. Li: \textit{Algebraic structures on Grothendieck groups of a tower of algebras}, J. Algebra \textbf{321} (2009), 2068--2084.

\bibitem{BMSW} N. Bergeron, S. Mykytiuk, F. Sottile, S. van Willigenburg: \textit{Shifted quasisymmetric functions and the Hopf algebra of peak functions}, Discrete Math. \textbf{246} (2002), 57--66.

\bibitem{BB} L. Billera, F. Brenti: \textit{Quasisymmetric functions and Kazhdan-Lusztig polynomials}, Israel J. Math. \textbf{184} (2011), 31--348.

\bibitem{BHW} L. Billera, S. Hsiao, S. van Willigenburg: \textit{Peak quasisymmetric functions and Eulerian
enumeration}, Adv. Math. \textbf{176} (2003), 248--276.

\bibitem{BH} S. Billey, M. Haiman: \textit{Schubert polynomials for the classical groups}, J. Amer. Math. Soc. \textbf{8} (1995), 443--482.

\bibitem{BC} F. Brenti, F. Caselli: \textit{Peak algebras, paths in the Bruhat graph and Kazhdan-Lusztig polynomials}, Adv. Math. \textbf{304} (2017), 539--582.

\bibitem{EG} K. Ebrahimi-Fard, L. Guo: \textit{Mixable shuffles, quasi-shuffles and Hopf algebras}, J. Algebraic Combin., \textbf{24} (2006), 83--101.

\bibitem{Eh} R. Ehrenborg: \textit{On posets and Hopf algebras}, Adv. Math. \textbf{119} (1996), 1--25.

\bibitem{GKL} I. Gelfand, D. Krob, A. Lascoux, B. Leclerc, V. Retakh, J.-Y. Thibon: \textit{Noncommutative
symmetric functions}, Adv. Math. \textbf{122} (1995), 218--348.

\bibitem{Ge} I. Gessel:
\textit{Multipartite {P}-partitions and inner products of skew {S}chur functions}, Contemp. Math. \textbf{34} (1984), 289--301.

\bibitem{Guo} L. Guo: \textit{An Introduction to Rota-Baxter algebra}, International Press, 2012.

\bibitem{GK} L. Guo, W. Keigher: \textit{Baxter algebras and shuffle products}, Adv. Math., \textbf{150} (2000), 117--149.

\bibitem{GTY} L. Guo, J.-Y. Thibon, H. Yu: \textit{Weak composition quasi-symmetric functions, Rota-Baxter algebras and Hopf algebras}, arXiv:1702.08011


\bibitem{JZ} R.-Q. Jian, J. Zhang: \textit{Rota-Baxter coalgebras}, arXiv:1409.3052

\bibitem{JL} N. Jing, Y. Li: \textit{A lift of Schur's Q-functions to the peak algebra}, J. Comb. Theory, Ser. A \textbf{135} (2015), 268--290.

\bibitem{KRY} J. P. S. Kung, G.-C. Rota, C. H. Yan: \textit{Combinatorics: the Rota Way}, Cambridge University
Press, 2009.

\bibitem{LMW} K. Luoto, S. Mykytiuk, S. van Willigenburg:
\textit{An introduction to quasi-symmetric Schur functions,
Hopf algebras, quasi-symmetric functions, and Young composition tableaux}, Springer Briefs in
Mathematics, Springer, New York, 2013.

\bibitem{ML} T. Ma, L. Liu: \textit{Rota-Baxter coalgebras and Rota-Baxter bialgebras},
Linear Multilinear Algebra  \textbf{64} (2016), 968--979.

\bibitem{MR} C. Malvenuto, C. Reutenauer:
\textit{Duality between quasi-symmetric functions and the {S}olomon descent algebra}, J. Algebra \textbf{177} (1995), 967--982.

\bibitem{Og} E. O$\breve{g}$uz: \textit{A note on Jing and Li's type B quasischur functions}, arXiv:1703.09358.

\bibitem{Ro} G.-C. Rota: \textit{Baxter algebras and combinatorial identities I$\&$II}, Bull. Amer. Math. Soc.,
\textbf{75} (1969), 325--329, 330--334.

\bibitem{Sch} M. Schocker:
\textit{The peak algebra of the symmetric group revisited}, Adv. Math. \textbf{192} (2005), 259--309.

\bibitem{Ste} J. Stembridge:
\textit{Enriched P-partitions},
Trans. Amer. Math. Soc. \textbf{349} (1997), 763--788.

\bibitem{YGZ} H. Yu, L. Guo, J. Zhao: \textit{Rota-Baxter algebras and left weak composition quasi-symmetric functions}, Ramanujan J. (2016), online available, DOI 10.1007/s11139-016-9822-0.

\end{thebibliography}

\end{document}